\documentclass[11pt]{amsart}

\usepackage{amsmath}
\usepackage{amssymb}
\usepackage{graphicx}

\usepackage{enumerate}

\newtheorem{theorem}{Theorem}[section]
\newtheorem
{cor}[theorem]{Corollary}
\newtheorem
{lem}[theorem]{Lemma}
\newtheorem
{prop}[theorem]{Proposition}
\theoremstyle{definition}
\newtheorem
{defn}[theorem]{Definition}
\newtheorem
{ex}[theorem]{Example}
\newtheorem
{rem}[theorem]{Remark}

\numberwithin{equation}{section}

\def\hypothesisname{Hypothesis}
\newtheorem{hypo}[theorem]{\hypothesisname}

\newenvironment{hyp}{\begin{hypo}\rm}{\end{hypo}}

\def\problemname{Problem}
\newtheorem{pro}[theorem]{\problemname}

\newenvironment{prob}{\begin{pro}\rm}{\end{pro}}

\renewcommand{\:}{\colon}
\newcommand{\1}{\mathbf{1}}
\newcommand{\ad}{\mathop{{\rm ad}}\nolimits}
\newcommand{\Ad}{\mathop{{\rm Ad}}\nolimits}
\newcommand{\aff}{\mathop{{\mathfrak{aff}}}\nolimits}
\newcommand{\Aff}{\mathop{{\rm Aff}}\nolimits}
\newcommand{\dd}{{\tt d}} 
\newcommand{\eps}{\varepsilon}
\newcommand{\Fix}{\mathop{{\rm Fix}}\nolimits}
\newcommand{\GL}{\mathop{{\rm GL}}\nolimits}
\newcommand{\id}{\mathop{{\rm id}}\nolimits}
\renewcommand{\Im}{\mathop{{\rm Im}}\nolimits}
\newcommand{\Ker}{\mathop{{\rm Ker}}\nolimits}
\newcommand{\la}{\langle}
\newcommand{\nin}{\noindent} 
\newcommand{\oline}{\overline}
\newcommand{\ra}{\rangle}
\renewcommand{\Re}{\mathop{{\rm Re}}\nolimits}
\newcommand{\res}{\vert}
\newcommand{\Spec}{{\rm Spec}}
\newcommand{\subeq}{\subseteq}

\newcommand{\de}{\mathrm{d}}
\newcommand{\ee}{\mathrm{e}}
\newcommand{\ie}{\mathrm{i}}

\newcommand{\cA}{\mathcal{A}}
\newcommand{\cC}{\mathcal{C}}
\newcommand{\cD}{\mathcal{D}}
\newcommand{\cH}{\mathcal{H}}
\newcommand{\cI}{\mathcal{I}}
\newcommand{\cK}{\mathcal{K}}
\newcommand{\cL}{\mathcal{L}}
\newcommand{\cO}{\mathcal{O}}
\newcommand{\cP}{\mathcal{P}}
\newcommand{\cS}{\mathcal{S}}
\newcommand{\cX}{\mathcal X} 
\newcommand{\cY}{\mathcal Y} 

\newcommand{\fa}{{\mathfrak a}}
\newcommand{\fg}{{\mathfrak g}}
\newcommand{\g}{{\mathfrak g}}
\newcommand{\fk}{{\mathfrak k}}
\newcommand{\fp}{{\mathfrak p}}

\newcommand{\sV}{{\tt V}}

\newcommand{\C}{{\mathbb C}}
\newcommand{\N}{{\mathbb N}}
\newcommand{\R}{{\mathbb R}}
\newcommand{\Z}{{\mathbb Z}}

\newcommand\bx{{\bf{x}}}

\title[Holomorphic extension of one-parameter operator groups]{Holomorphic extension of one-parameter operator groups}

\author[D. Belti\c t\u a]{Daniel Belti\c t\u a}
\address[D. Belti\c t\u a]{Institute of Mathematics ``Simion Stoilow'' of the Romanian Academy,
		P.O. Box 1-764, Bucharest, Romania} 
\email{\tt Daniel.Beltita@imar.ro, beltita@gmail.com}

\author[K.-H. Neeb]{Karl-Hermann Neeb}
\address[K.-H. Neeb]{Department Mathematik, Friedrich-Alexander-Universit\"at \break Erlangen-N\"urnberg,
		Cauerstrasse 11, 91058 Erlangen, Germany} 
\email{\tt neeb@math.fau.de}

\keywords{analytic extension, one-parameter group of operators, KMS condition,
	standard subspace, distribution vector,
	analytic vector}

\subjclass[2010]{Primary 22E45; Secondary 46A20, 47D06}

\begin{document}

\begin{abstract}
We study holomorphic extensions of one-parameter groups on locally convex spaces with a view to applications to KMS boundary conditions. In the first part we deal with analytic extensions of one-parameter groups of operators on locally convex spaces and in the second part we apply our results to spaces of distribution vectors of unitary representations of Lie groups. This leads to new tools that can be used to construct, from unitary Lie group
representations, nets of standard subspaces, as they appear in Algebraic Quantum Field Theory. We also show that these methods fail for spaces of analytic vectors, and this in turn leads to new maximality results for domains of analytic extensions of orbit maps for
unitary representations.
\end{abstract}

\maketitle

\tableofcontents


\section{Introduction}
\label{Sect1}

One-parameter groups of linear operators are an important tool in a wide variety of
mathematical areas such as Functional Analysis, Differential Geometry,
Probability Theory, or Mathematical Physics. 
In the present paper we study holomorphic extensions of one-parameter groups on locally convex spaces with a view to applications to Kubo--Martin--Schwinger (KMS) boundary conditions that are involved in the construction of standard subspaces in the framework of unitary representations of Lie groups and their homogeneous spaces (cf.\ \cite{FNO23}). Here the KMS conditions arise
from a one-parameter group $(U_t)_{t \in \R}$ on a complex topological vector space
$\cY$ and an antilinear operator $J$ on $\cY$ commuting with each operator $U_t$.
Then the subspace $\cY_{\rm KMS}$ of interest consists of all elements
$y \in \cY$ whose orbit map $U^v \colon \R \to \cY$, $t \mapsto U_t v$
extends to a continuous map on  
$\oline{\cS}_{0,\pi}:=\R+\ie[0,\pi]$, weakly
holomorphic on the interior $\cS_{0,\pi}$,  such that
\begin{equation}
	\label{eq:kms}
	U^v(\pi\ie + t) = J U^v(t) = J U_t v\quad \mbox{ for }\quad t \in \R.
\end{equation}

Along with Quantum Statistical Mechanics, 
the standard forms of von Neumann algebras and the Tomita--Takesaki modular theory have motivated the interest in KMS boundary conditions
(\cite{BR87, Ta03}). 
Specifically, it early turned out that the modular automorphism group associated to a faithful normal semifinite weight of a $W^*$-algebra
is determined by the KMS condition (\cite[Cor.~4.9]{Co71}). 
Analytic vectors with respect to the modular groups were then used in
\cite[Sect.~3]{PT73}, in order to establish noncommutative Radon--Nikodym
theorems, and the holomorphic extension operators 
in the framework provided by a duality pairing of Banach spaces were later systematically studied in \cite{CZ76} and \cite{Zs77},
where it was proved for instance that the extension operators are closed and densely defined.  
However, the applications of these results to one-parameter groups on spaces of smooth vectors and distribution vectors associated to unitary representations of Lie groups are quite limited, due to the fact that these spaces are not Banach spaces in general. 
This motivates the development of a general theory of 
holomorphic extensions of one-parameter groups 
on locally convex spaces. 

This paper consists of two parts. The first abstract part,
consisting of Sections~\ref{Sect2}-\ref{Sect4},  
deals with analytic extensions of one-parameter groups of operators on
locally convex spaces and the second part,
Sections~\ref{Sect5}-\ref{Sect7}, 
applies some key results of this theory to spaces of distribution vectors
of unitary representations. 

The content of this paper is as follows:
Section~\ref{Sect2} contains the definition of the notion of holomorphic extension of a one-parameter group $(U_t)_{t \in \R}$
of continuous linear operators on a Hausdorff locally convex space $\cY$ 
(Definition~\ref{holomext_def}). In Section~\ref{Sect3},  
we develop an operator calculus for one-parameter groups
on 
sequentially complete locally convex spaces,  
satisfying suitable continuity and growth conditions.
For simplicity, we subsume all these conditions under the term
{\it exponential growth} (Definition~\ref{def:expgro}).
We derive some important facts, notably density properties of the corresponding space of entire vectors  (Corollary~\ref{agrowth_cor}).  
Section~\ref{Sect4} contains our main results on holomorphic extensions of
one-parameter groups of exponential growth on
a suitable class of locally convex spaces, including all Fr\'echet spaces. 
Specifically, we prove that the holomorphic extensions of a one-parameter
group $(U_t)_{t \in \R}$ and of its antidual group, given by
$V_t := U_t^\sharp$ on the antidual space $\cY^\sharp$, are antidual to each other
in the sense that we have the
Duality Theorem~\ref{UtoV}: 
\[ (U_z)^\sharp = V_{\oline z} \quad \mbox{ for } \quad z \in \C.\]
Here $U_z^\sharp$ is the antidual of the operator $U_z$ obtained by
analytic continuation from $U$, and $V_{\oline z}$ is likewise obtained
from~$V$. Proposition~\ref{VtoU} further connects the antidual
of $V_{\oline z}$ with the double antidual of $U_z$.
This does not follow from the Duality Theorem because
$\cY^\sharp$, endowed with the weak-$*$-topology, need not 
belong to the class of locally convex spaces we are working with. 
As an important consequence, 
the operators obtained by holomorphic extension are closed (cf.~also
Theorem~\ref{F-tech}). 
These results extend  \cite[Th.~1.1]{Zs77} and
\cite[Th.~2.4]{CZ76} 
beyond the setting of Banach spaces. 

In Section~\ref{Sect5} we then turn to the applications to
unitary representations $(U,\cH)$ of a Lie group~$G$ with
Lie algebra~$\g$. The short Section~\ref{Sect5} 
has a rather technical character and its aim is to
prepare the ground for applications of the abstract results of
Sections~\ref{Sect3} and \ref{Sect4}.
Concretely, we fix a Lie algebra element $h \in \g$ 
and show that the action of the one-parameter group $U_{h,t} := U(\exp th)$
on the Fr\'echet space~$\cH^\infty$ of smooth vectors
is of exponential growth. We are particularly 
interested in its antidual one-parameter group 
on the space $\cH^{-\infty} := (\cH^\infty)^\sharp$
of {\it distribution vectors}.
In Section~\ref{Sect6}, we study 
the space $\cH^{-\infty}_{\rm KMS} = (\cH^{-\infty})_{\rm KMS}$ specified
in terms of analytic extension by the KMS condition~\eqref{eq:kms}. 
Here the boundary condition is defined in terms of an antilinear involutive isometry $J\colon \cH\to\cH$ satisfying  $J(\cH^\infty)\subseteq\cH^\infty$ and
commuting with the unitary 
one-parameter group
\[ U_{h,t} = U(\exp th) = \ee^{t \partial U(h)}.\] 
We write
\[ \sV:=\{\xi\in \cD(\Delta_\sV^{1/2}) \colon \Delta^{1/2}\xi=J\xi\}
\quad \mbox{ with } \quad
\Delta_\sV := \ee^{2\pi \ie \partial U(h)}, \]
for the standard subspace of $\cH$ corresponding to the pair of
modular objects $(\Delta_\sV,J)$.
Recall that a closed real subspace $\sV \subeq \cH$ is said to
be {\it standard} if $\sV \cap i \sV = \{0\}$ and
$\sV + i \sV$ is dense in $\cH$ (cf.\ \cite{Lo08}). 
In this setting, our main results are: 
\begin{itemize}
\item  $\cH^{-\infty}_{\rm KMS}$  is a weak-$*$-closed subspace of $\cH^{-\infty}$ (Theorem~\ref{extcl}); \phantom{\Big)}
\item  $\cH^{-\infty}_{\rm KMS}\cap\cH=\sV$ (Theorem~\ref{thm:KMSV});\phantom{\Big)}
\item $\sV$ is weak-$*$ dense in $\cH^{-\infty}_{\rm KMS}$ (Theorem~\ref{thm:vdense}). 
\end{itemize}  

These results on analytic continuation of $U_h$-orbit maps
for distribution vectors are complemented by the discussion
in Section~\ref{Sect7}, where we focus on the dense subspace
$\cH^\omega$ of analytic vectors of a unitary representation $(U,\cH)$ of a
Lie group~$G$. The space $\cH^\omega$ carries a natural 
topology as a  locally convex direct limit, and its antidual
$\cH^{-\omega}$ is a Fr\'echet space
(see \cite{FNO23} and \cite{GKS11} for more details on these topologies).
Again, every $h \in \g$ defines one-parameter groups
$(U^{\pm \omega}_{h,t})_{t \in \R}$ on the spaces $\cH^{\pm \omega}$, but
our discussion reveals that these one-parameter groups 
are not of exponential growth.
We start in Section~\ref{Sect7.1} by
recalling some well-known facts on the Maximum Modulus Theorem
on strips. This is used in Section~\ref{Sect7.2} to show that,
for the affine group $G = \Aff(\R)_e$  of the real line, 
the orbit maps of the dilation group in $\cH^\omega$ never extend to
strips of width exceeding~$\pi$ (Theorem~\ref{thm:ax+b-gen}).
As non-compact semisimple Lie groups 
contain many copies of the group $\Aff(\R)_e$,
restriction to such subgroups yields  natural ``upper bounds''
on the domains to which orbit maps of analytic vectors
could extend analytically. The corresponding conclusions
are formulated in Theorem~\ref{thm:7.9}, which
improves Goodman's results stated in \cite{Go69}.
That our  results are optimal
follows from existence results on
orbit maps by Kr\"otz and Stanton for semisimple Lie
groups (\cite{KSt04}).
We conclude this paper with a brief discussion of
maximal analytic extensions of orbit maps in $\cH^\omega$
for general Lie groups (Section~\ref{Sect7.4}).

This article was motivated by
the connections between causal structures
on homogeneous spaces, Algebraic Quantum Field Theory (AQFT), 
modular theory of operator algebras 
and unitary representations of Lie groups
(cf.~\cite{BGL02, NO21, MN21, MNO23, NO23}).
In this context one has to specify certain 
standard subspaces $\sV\subeq \cH$ associated to
unitary representations $(U,\cH)$ of Lie groups.
Here $\sV$ is 
$\cH_{\rm KMS}$ with respect to a unitary
one-parameter group $U_{h,t} = U(\exp th)$ for some $h \in \g$,
and some conjugation $J$ commuting with this one-parameter group.
Typically $h$ is a so-called {\it Euler element}, i.e.,
$\ad h$ is diagonalizable with $\Spec(\ad h) \subeq \{ 1,0,-1\}$
(cf.~\cite{MN23}).
Our Theorem~\ref{thm:7.12} implies that, typically,
$(\cH^\omega)_{\rm KMS} = \{0\}$, so that one should not expect
elements of $\sV$ to be analytic vectors.
They  behave rather badly from the algebraic perspective,
but this makes them interesting for localization
purposes (which is incompatible with analyticity of orbit maps).
It turns out that natural ``generators'' of $\sV$ can be found
in the space $\cH^{-\infty}$ of disctribution vectors.
The results in Section~\ref{Sect6} provide very effective
tools to identify such generating sets.
They have already been used in \cite{FNO23} to construct
interesting nets of standard subspaces for representations
of semisimple Lie groups. To complement these results on the class of 
semisimple Lie groups, we plan to explore in 
\cite{BN23} possible analogs for solvable Lie groups.\\

\nin{\bf Acknowledgment:}
We thank Jonas Schober and Tobias Simon for abundant comments
on a first draft of this paper.
We also thank the referee for an extremely constructive report
that led to various improvements of the exposition and 
to simplifications of the assumptions of some of our results.
In particular the material in the appendices has been suggested
  by the referee.

The research of D.~Belti\c{t}\u{a} was supported by a grant of the Ministry of Research, Innovation and Digitization, CNCS/CCCDI –UEFISCDI, project number PN-III-P4-ID-PCE-2020-0878, within PNCDI III. 
The research of K.-H. Neeb was partially supported
by DFG-grant NE 413/10-2.

\section{Holomorphic extension of one-parameter groups}
\label{Sect2}

In this section we introduce holomorphic extensions of a
one-parameter group $(U_t)_{t\in \R}$ of operators on a locally convex space
$\cY$ (Definition~\ref{holomext_def}). Our extensions will be
based on analytic extension of orbit maps to strips and also maximal
in a suitable sense. This will lead to a family $(U_z)_{z \in \C}$
of operators on $\cY$ that are, under suitable assumptions, densely
defined. However, if a one-paramter group
is not of exponential growth, 
we shall see in Section~\ref{Sect7} that it may happen that
the operator~$U_z$ becomes trivial (in the sense that its domain
is $\{0\}$) when $\vert\Im z\vert$ is sufficiently large
(Theorem~\ref{thm:ax+b-gen}). 
We also establish a couple of basic properties that only require the ambient locally convex space to be a Hausdorff space.

For every locally convex space $\cY$ over $\C$ we introduce the
corresponding spaces and structures:
\begin{itemize}
	\item $\cY^\sharp$  is the space of
	continuous antilinear functionals, endowed with the
	weak-$*$-topology, with the corresponding sesquilinear
	antiduality pairing 
	$\langle\cdot,\cdot\rangle\colon\cY\times\cY^\sharp\to\C$. 
	\item $\cY^\sharp_c$ is the space $\cY^\sharp$,
	endowed with the topology of uniform convergence on compact subsets of~$\cY$.
	\item  $\cY_w$ is the space $\cY$, endowed its weak topology.  
\end{itemize}
In particular, $\cY^\sharp=(\cY^\sharp)_w$ as topological vector spaces. 

The algebra of continuous linear operators on $\cY$ is denoted
by~$\cL(\cY)$, and $\GL(\cY)$ is its group of invertible elements. 
If $T\colon\cD(T)\subseteq\cY\to \cY$ is a linear operator whose domain $\cD(T)$ is dense in~$\cY$, 
then we define $\cD(T^\sharp)$ as the set of all functionals $\eta\in\cY^\sharp$ for which 
there exists $T^\sharp\eta\in\cY^\sharp$ with  $T^\sharp\eta\vert_{\cD(T)}=\eta\circ T$. 
Since  $\cD(T)$ is dense in $\cY$, it follows that $T^\sharp\eta$ is uniquely determined by $\eta$, and we thus obtain a new linear operator $T^\sharp\colon\cD(T^\sharp)\to\cY^\sharp$. 
For every $T\in\cL(\cY)$ we  have $T^\sharp\in\cL(\cY^\sharp)$.

\begin{rem}
	\label{adjgraph}
	For any densely-defined linear operator
	$T\colon \cD(T)\subseteq \cY\to\cY$, we denote
	its graph by $\Gamma(T)=\{(y,Ty): y\in\cD(T)\}\subseteq\cY\times\cY$.
	We consider the map 
	\[ S\colon \cY\times\cY\to \cY\times\cY,\qquad S(y,v):=(-v,y). \]
	For every $(\eta,\zeta)\in\cY^\sharp\times\cY^\sharp$ we have 
	\[(\eta,\zeta)\in\Gamma(T^\sharp)\iff (\forall y\in\cD(T))\ \langle Ty,\eta\rangle =\langle y,\zeta\rangle 
	\iff (\eta,\zeta)\in S(\Gamma(T))^\perp, 
	\]
	hence 
	\begin{equation}
		\label{adjgraph_eq1}
		\Gamma(T^\sharp)=S(\Gamma(T))^\perp\subseteq \cY^\sharp\times\cY^\sharp
	\end{equation}
	where the annihilator refers to the antiduality pairing 
	\[ \langle\cdot,\cdot\rangle\colon (\cY\times\cY)\times(\cY^\sharp\times\cY^\sharp)\to\C,\quad
	\langle(y,v),(\eta,\zeta)\rangle
	:=\langle y,\eta\rangle+\langle v,\zeta\rangle.\] 
	It follows by \eqref{adjgraph_eq1} that $\Gamma(T^\sharp)$ is closed in $\cY^\sharp\times\cY^\sharp$, hence also in $\cY^\sharp_c\times\cY^\sharp_c$.
\end{rem}

\begin{rem}
	\label{lintop}
	We recall a few facts on general Functional Analysis for later use: 
	\begin{enumerate}[{\rm(i)}]
		\item\label{lintop_item1}
		The topology of a Hausdorff locally convex space $\cY$ is
		identical with the topology of uniform convergence on every equicontinuous subset of $\cY^\sharp$ (\cite[Prop.~36.1]{Tr67}).
		\item\label{lintop_item2} For any locally convex spaces $\cX$ and $\cY$, the topology of pointwise convergence and compact convergence
		coincide on equicontinuous subsets of the space of continuous linear operators from $\cX$ to $\cY$
		(\cite[Prop.~32.5]{Tr67}). 
		\item\label{lintop_item3}
		If $\cY$ is a Fr\'echet space then $\cL(\cY_w)\subseteq\cL(\cY)$ by the Closed Graph Theorem.  
	\end{enumerate}
\end{rem}

\begin{defn}
  Let $\cX$ be a Hausdorff locally convex space and $\Omega\subseteq\C$
an open subset with its closure denoted by~$\overline{\Omega}$.
	We write $\cO^w_{\partial}(\oline\Omega, \cX)$
        for the space of $\cX$-valued
	continuous functions $f \: \oline\Omega \to \cX$ 
        whose restriction to the interior
	of $\Omega$ is weakly holomorphic, i.e,
        for each continuous linear functional $\alpha$ on $\cX$, the function
        $\alpha \circ f$ is holomorphic on~$\Omega$.
       	We write $\cO_{\partial}(\oline\Omega, \cX)$ for the subspace
        of those functions functions on $\oline\Omega$ that are
        holomorphic on $\Omega$. 
        We refer to Remark~\ref{rem:a.1} for a discussion of this concept
        and comparison with holomorphy.
\end{defn}

\begin{rem}
	\label{duality}
	Let $\Omega \subeq \C$ be an open subset.
	As $\cY^\sharp$ carries the weak-$*$-topology
	with respect to $\cY$,
	a function $f \colon  \Omega \to \cY^\sharp$
	is weakly holomorphic if and only if
	all compositions $z \mapsto f(z)(\xi)$, $\xi \in \cY$, 
	with continuous linear functionals are holomorphic. 
       If, in addition, $\cY$ is barreled, 
        then $\cY^\sharp$ is sequentially complete
	by Remark~\ref{rem:a.3}, hence Remark~\ref{rem:a.1}   
	shows that weakly holomorphic maps into $\cY^\sharp$
	are actually holomorphic.  
	
        If $\cY$ is Fr\'echet, then
        $\cY^\sharp$ and $\cY^\sharp_c$ have the same dual space
        (Remark~\ref{surjev}\eqref{surjev_item2}), given by point evaluations,
     hence  any $\cY^\sharp$-valued holomorphic map is weakly holomorphic as
      a $\cY^\sharp_c$-valued map. 
	Using the fact that $\cY_c^\sharp$ is sequentially complete 
	(Remark~\ref{rem:a.2}), 
	it follows from 
	Remark~\ref{rem:a.1} again
	that $f$ is also holomorphic as a
	$\cY^\sharp_c$-valued map.
\end{rem}

For $\varepsilon_1\le0\le\varepsilon_2$, we write 
$$\cS_{\varepsilon_1,\varepsilon_2}:=\{z\in\C\colon 
\varepsilon_1<\Im z< \varepsilon_2\}$$
and denote the closure of this strip by
$\overline{\cS}_{\varepsilon_1,\varepsilon_2}$.
We also abbreviate $\cS_{\pm r} := \cS_{-r,r}$ for $r > 0$. 

\begin{defn} {\rm(Analytic extension of one-parameter groups)}
	\label{holomext_def}
	Let $\cY$ be a Hausdorff locally convex space 
	and   $(U_t)_{t\in\R}$  a one-parameter group in $\GL(\cY)$.  
	For  $\varepsilon_1\le0\le\varepsilon_2$, we consider the following
	linear subspace of $\cY$
	\[	\cY_{\varepsilon_1,\varepsilon_2}:=
	\{y\in\cY:(\exists F_y \in
	\cO^w_{\partial}(\overline{\cS}_{\varepsilon_1,\varepsilon_2}, \cY)) (\forall t\in\R)\quad F_y(t)=U_ty\}\] 
	and 
	\[	\cY^\cO
	:=
	\{y\in\cY:(\exists F_y \in
	\cO^w(\C, \cY))(\forall t\in\R)\quad F_y(t)=U_ty\}.\] 
	For every $z\in\C$ we also define $\cD(U_z)\subseteq\cY$ and $U_z\colon\cD(U_z)\to\cY$ by 
	\[
	\cD(U_z)
	=\begin{cases}
		\cY_{0,\Im z}&\text{ if }\Im z> 0,\\
		\cY&\text{ if }\Im z=0, \\
		\cY_{\Im z,0}&\text{ if }\Im z<0,
	\end{cases}
	\mbox{ and }
	U_z y:=F_y(z) \mbox{  for all }  y\in\cD(U_z).\]
	Note that, for $z=s\in\R$, this definition reproduces the operator~$U_s$. 
	The operator family $(U_z\colon\cD(U_z)\to\cY)_{z\in\C}$ is called the \emph{holomorphic extension of the one-parameter group}~$(U_t)_{t\in\R}$ and the elements of the linear subspace 
	$\cY^\cO$ 
	are called its \emph{weakly entire vectors}.  
\end{defn}

\begin{rem} (a) 
	In the above definition, the uniqueness of the extension  $F_x$
	follows by a straightforward application of the Riemann--Schwarz
	symmetry principle (cf.~\cite[no.~42, Lemme]{Ch90}).
	
	\nin (b) Theorem~\ref{thm:ax+b-ex} below shows that there exist natural
	examples of one-parameter groups for which the domain
	of $U_z$ is trivial for $|\Im z| > \pi/2$.         
\end{rem}

\begin{lem}
	\label{holomext_fact}
	If $z = a + \ie  b\in \overline{\cS}_{\varepsilon_1,\varepsilon_2}$, and
	$y\in\cY_{\varepsilon_1,\varepsilon_2}$, then 
	\[ U_z y\in\cY_{\varepsilon_1-b,\varepsilon_2-b} \quad \mbox{ and } \quad
	U_w U_z y=U_{w+z}y \quad \mbox{ for } \quad
	w\in\cS_{\varepsilon_1-b,\varepsilon_2-b}.\]
\end{lem}

\begin{proof}
	Since $y\in \cY_{\varepsilon_1,\varepsilon_2}$, we have
	$F_y \in \cO_{\partial}(\overline{\cS}_{\varepsilon_1,\varepsilon_2},\cY)$
	with $F_y(s)=U_sx$ for all $s\in\R$.  
	
	For each $t\in\R$ both functions $U_tF_y(\cdot)$ and $F_y(t+\cdot)$ are
	continuous on $\overline{\cS}_{\varepsilon_1,\varepsilon_2}$,
	weakly holomorphic on the interior, 
	and their restrictions to $\R$ are equal to the function
	$s\mapsto U_tU_sy=U_{t+s}y$.  
	Hence, for every $z\in\overline{\cS}_{\varepsilon_1,\varepsilon_2}$ and $t\in\R$, we have 
	$U_tF_y(z)=F_y(t+z)$, that is, 
	\begin{equation}\label{holomext_fact_prf_eq1}
		U_tU_zy=F_y(t+z). 
	\end{equation} 
	This shows that, for $z\in\overline{\cS}_{\varepsilon_1,\varepsilon_2}$, 
	the function $\cS_{\varepsilon_1-b,\varepsilon_2-b}\to\cY$, 
	$w\mapsto F_y(w+z)$ is an $\cO^w_\partial$-extension of  
	the orbit map $\R\to\cX$, $t\mapsto U_tU_zy$. 
	Consequently $U_z x\in\cY_{\varepsilon_1-b,\varepsilon_2-b}$ 
	and $U_w U_zy=U_{w+z}y$ for all 	$w\in\cS_{\varepsilon_1-b,\varepsilon_2-b}$.
\end{proof}

\begin{prop}
	\label{holomext_prop}
	The following assertions hold: 
	\begin{enumerate}[{\rm(i)}]
		\item\label{holomext_prop_item1}
		For every $z\in\C$, the linear operator $U_z\colon\cD(U_z)\to\cY$ is injective, 
		its image equals $\cD(U_{-z})$, 
		and we have $U_z^{-1}=U_{-z}\colon \cD(U_{-z})\to\cY$. 
		\item\label{holomext_prop_item2} 
		If $z,w\in\C$ and $\Im z$ and $\Im w$ have the same sign,
		then $U_wU_z=U_{w+z}$. 
	\end{enumerate}
\end{prop}

\begin{proof}
	\eqref{holomext_prop_item1}
	Let $y\in\cD(U_z)$ with $U_zy=0$. 
	We apply Lemma~\ref{holomext_fact} for the pair $(\varepsilon_1,\varepsilon_2)=(0,\Im z)$ if $\Im z\ge 0$, and $(\varepsilon_1,\varepsilon_2)=(\Im z,0)$  if $\Im z\le 0$. 
	For $w:=t\in\R$ we then obtain $U_{t+z}y=U_tU_zy=0$. 
	Thus the function
	$F_y\in \cO_\partial^w(\overline{\cS}_{\varepsilon_1,\varepsilon_2},\cY)$ 
	satisfies $F_y(z+t)=0$ for every $t\in\R$, hence $F_y=0$. 
	Therefore $y=F_y(0)=0$. 
	
	If we apply Lemma~\ref{holomext_fact} for the pair $(\varepsilon_1,\varepsilon_2)$ as above and $w:=-z$, for every $y\in\cD(U_z)$ we obtain $U_zy\in\cD(U_{-z})$ and $U_{-z}U_zy=y$. 
	That is, we have the inclusion of unbounded operators $U_z\subseteq U_{-z}^{-1}$, 
	which is equivalent to $U_z^{-1}\subseteq U_{-z}$. 
	Replacing $z$ by $-z$, we obtain the assertion. 
	
	\nin \eqref{holomext_prop_item2}
	For all $z_0,w_0\in\C$ with $\Im z_0$ and $\Im w_0$ having the same sign, 
	we must prove that 
	\begin{equation}
		\label{holomext_prop_proof_eq1}
		\cD(U_{w_0+z_0})=\{y\in\cD(U_{z_0}): U_{z_0}y\in\cD(U_{w_0})\}
	\end{equation} and $U_{w_0}U_{z_0}y=U_{w_0+z_0}y$ for every $y\in\cD(U_{w_0+z_0})$. 
	By~\eqref{holomext_prop_item1}
	we may w.l.o.g.\ assume  that $0\le \Im z_0$ and $0\le \Im w_0$. 
	
	An application of  Lemma~\ref{holomext_fact} for the pair
	$(\varepsilon_1,\varepsilon_2)=(0,\Im w_0+\Im z_0)$ and $y\in\cD(U_{w_0+z_0})$ 
	shows that $U_{z_0}y\in \cY_{-\Im z_0,\Im w_0}\subseteq\cY_{0,\Im w_0}=\cD(U_{w_0})$ and $U_{w_0}U_{z_0}y=U_{w_0+z_0}y$. 
	This proves $\subseteq$ in \eqref{holomext_prop_proof_eq1}.
	
	Conversely, let $y\in\cD(U_{z_0})$ with $U_{z_0}y\in\cD(U_{w_0})$. 
	By the inclusion $\subseteq$ in \eqref{holomext_prop_proof_eq1}, we have $y\in\cD(U_{z_0})\subseteq\cD(U_{z})$ if $0\le\Im z\le\Im z_0$ 
	and 
	$U_{z_0}y\in\cD(U_{w_0})\subseteq\cD(U_w)$ if $0\le \Im w\le \Im w_0$. 
	Hence we may define 
	\begin{align*}
		& F  \colon
	\overline{\cS}_{0,\Im(w_0+z_0)}
	\to\cY, \\
	\quad 
	& F(z) :=\begin{cases}
		F_{U_{z_0}y}(z-z_0)=U_{z-z_0}U_{z_0}y& \text{ if }\Im z_0\le \Im z\le \Im(w_0+z_0),\\
		F_y(z)=U_zy  & \text{ if }0\le \Im z\le \Im z_0, 
	\end{cases}
\end{align*}
	This function is well-defined, because, for
	$\Im z=\Im z_0$ and ${s:=z-z_0}$, the relation
	$U_sU_{z_0}y=U_zy$ by Lemma~\ref{holomext_fact}. 
	Since $F_{U_{z_0}y}\in \cO^w_\partial(\overline{\cS}_{0,\Im w_0},\cY)$ and 
	$F_y\in\cO^w_\partial(\overline{\cS}_{0,\Im z_0},\cY)$, 
	the function $F$ is continuous. 
	Moreover, $F$ is weakly holomorphic on the open strip $\cS_{0,\Im (w_0+z_0)}$ by an application of \cite[no.~42, Lemme]{Ch90}. 
	Thus $F\in\cO_\partial^w(\overline{\cS}_{0,\Im (w_0+z_0)},\cY)$  
	and, taking into account that $F(t)=U_ty$ for every $t\in\R$, we obtain $y\in\cD(U_{w_0+z_0})$. 
	This completes the proof of \eqref{holomext_prop_proof_eq1}. 
\end{proof}

\section{Operator calculus 
  with functions of Gaussian growth}
\label{Sect3}

Assuming, in addition, that the ambient Hausdorff
locally convex space $\cY$ is \emph{sequentially} complete,
i.e., every 
Cauchy sequence 
is convergent, the exponential
growth condition~\eqref{agrowth_eq1} in Definition~\ref{def:expgro} below 
becomes strong enough to allow operator calculus with
Gaussian functions on $\R$. This allows us in particular
to approximate general vectors by entire ones
(Corollary~\ref{agrowth_cor}) and to
establish some approximation properties of holomorphic extensions
of one-parameter groups (Corollary~\ref{F-tech_cor2}).
Here the role of sequential completeness is to ensure the existence
of weak integrals 
of functions of one real variable 
in order to develop a functional calculus for one-parameter groups satisfying suitable continuity and growth conditions
(Proposition~\ref{agrowth}). 

\begin{lem}
	\label{impr}
	{\rm(Vector-valued improper integrals)}
	Let $\cY$ be a sequentially complete,
	Hausdorff, locally convex space. 
	We write $L^1_c(\R,\cY)$ for the space of continuous functions
	$f \colon \R\to \cY$ such that,
	for every continuous seminorm $p$ on $\cY$
	we have
	$\int_\R p(f(t))\de t<\infty$.
	Then the following assertions hold for $f \in L^1_c(\R,\cY)$: 
	\begin{enumerate}[{\rm(i)}]
		\item\label{impr_item1} 
		The weak improper integral $\int_\R f(t)\de t$
		exists in~$\cY$.
		It is uniquely determined by
		$\int_\R \alpha(f(t))\de t = \alpha\big(\int_\R f(t)\de t\big)$
		for all $\alpha \in \cY^\sharp$ and coincides with
		$\lim_{n\to\infty}\int_{[-n,n]} f(t)\de t$ 
		\item\label{impr_item2} 
		We have
		$p\big(\int_\R f(t)\de t\big)\le\int_\R p(f(t))\de t$ for every continuous seminorm~$p$ on $\cY$. 
	\end{enumerate}	
\end{lem}

\begin{proof}
	\eqref{impr_item1} 
	For every $n \in \N$, the integral
	$I_n := \int_{-n}^n \ f(t)\de t = \int_{[-n,n]} f(t)\, \de t$
	exists in $\cY$ because $\cY$ is sequentially complete and the
	integral can be obtained as a limit of Riemann sum
	(cf.~\cite[Prop.~1.2.3]{He89}).
	We have to show that the sequence $(I_n)_{n \in \N}$
	converges in~$\cY$.
	
	Let $\varepsilon>0$ be
	arbitrary and $p$ be any continuous seminorm on $\cY$. 
	Since $\int_\R p(f(t))\de t<\infty$, there exists $N\in\N$ 
	such that 
\[ p\Big(\int_{\R \setminus [-N,N]} f(t)\de t\Big)\le\int_{\R \setminus [-N,N]}
	p(f(t))\de t<\varepsilon.\] 
	(See for instance \cite[Lemma 1.1.9]{GN} for the first of these inequalities.)
	Then, for any $N < n \leq m$, 
	we have 
	\[p(I_m - I_n) 
	=p\Bigl(\int_{[-m,-m] \setminus [-n,n]} f(t)\de t\Bigr)\le \varepsilon\]
	hence $(I_n)_{n \in \N}$  is a Cauchy net, hence convergent.
	It follow that 
        \[ \int_\R f(t)\de t:=\lim_{n \to \infty} I_n \]
        exists in~$\cY$. 
	
	
	\nin \eqref{impr_item2}
	For every continuous seminorm $p$ on $\cY$ and every 
	$n \in \N$, we have 
	\[ p\Bigl(\int_{[-n,n]} f(t)\de t\Bigr)\le\int_{[-n,n]} p(f(t))\de t
	\le\int_\R p(f(t))\de t.\]  
	Passing to limit in the above inequality, we get 
	$p\big(\int_\R f(t)\de t\big)\le\int_\R p(f(t))\de t$
	for every continuous seminorm~$p$ on $\cY$.
	Finally, the universal property of the weak integral
	is obtained by a passage to the limit.
\end{proof}

The operator calculus with Gaussian functions, prepared by the following
lemma, will be used below to obtain a rich supply of entire vectors. 
In the following, we write $\cC(X,Y)$ for the set of continuous
functions $f : X \to Y$.

\begin{lem}
	\label{lem:gaussians}
	For every function $\phi \colon \R \to \C$, we define
	\[ p_a(\phi) :=  \sup \{ |\varphi(t)| \ee^{t^2/a} \colon t \in \R\}\in [0,\infty] \]
	and
	\[\cA = \bigcup_{a > 0} \cA_a \subeq \cC(\R,\C), \quad \mbox{ where } \quad
	\cA_a := \{ \phi  \in \cC(\R,\C)\colon p_a(\phi) < \infty\}.\]
	We consider the Gaussians 
	\[ \gamma_{a,z}\colon \R\to\C, \quad
	\gamma_{a,z}(t):=\frac{1}{\sqrt{\pi a}}\ee^{-(t-z)^2/a}
	\quad \mbox{ for }  \quad a > 0, z\in\C.\] 
	Then the following assertions hold: 
	\begin{enumerate}[\rm(i)] 
		\item
		\label{lem:gaussians_item1}
		$\gamma_{a,z} * \gamma_{b,w} = \gamma_{a+b,z + w}$ for
		$a,b \in \R, z,w \in \C$. 
		
		\item
		\label{lem:gaussians_item2}
		$p_b \leq p_a$ for $a \leq b$ and
		\[ p_{a_1 + a_2}(\phi_1 * \phi_2)
		\leq \frac{\sqrt{\pi} \sqrt{a_1 a_2}}{\sqrt{a_1 + a_2}}
		p_{a_1}(\phi_1) p_{a_2} (\phi_2).\] 
		\item
		\label{lem:gaussians_item3}
		$\cA$ is a $*$-subalgebra of the convolution
		Banach $*$-algebra $L^1(\R)$
		and the inclusions of the normed spaces
		$(\cA_a,p_a) \to L^1(\R)$ are continuous. 
		\item
		\label{lem:gaussians_item4}
		$\gamma_{a,z}\in\cA_b$ for $0 < a < b, z \in\C$,  
		and the maps $\C \to (\cA_b, p_b), z \mapsto \gamma_{a,z}$,
		are holomorphic. 
	\end{enumerate}
\end{lem}

\begin{proof} 
	\eqref{lem:gaussians_item1}
	This relation easily reduces to
	$\gamma_{a,0} * \gamma_{b,0} = \gamma_{a+ b,0}$, which is well-known.
	
	\nin 
	\eqref{lem:gaussians_item2}
	Clearly, $p_b \leq p_a$ for $a \leq b$.
	If $c_j := p_{a_j}(\phi_j) < \infty$, then
	$\phi_j \leq c_j \sqrt{\pi a_j} \gamma_{a_j,0}$ implies
	\[ |(\phi_1 * \phi_2)(x)| \leq c_1 c_2 \sqrt{\pi a_1} \sqrt{\pi a_2}
	(\gamma_{a_1,0} * \gamma_{a_2,0})(x) 
	=c_1 c_2 \sqrt{\pi a_1} \sqrt{\pi a_2} \gamma_{a_1+a_2,0}(x),\]
	so that
	\[ p_{a_1+a_2}(\phi_1 * \phi_2) \leq p_{a_1}(\phi_1) p_{a_2}(\phi_2)
	\frac{\sqrt{\pi a_1}     \sqrt{\pi a_2}}{\sqrt{\pi (a_1 + a_2)}}.\]
	
	\nin 
	\eqref{lem:gaussians_item3}
	It is clear that $\cA$ is a linear subspace of $L^1(\R)$ and for every $\varphi\in\cA$
	the function defined by
	$\varphi^*(x) := \oline{\phi(-x)}$ is also contained 
	in $\cA$. 
	If $\varphi_1,\varphi_2\in\cA$, then 
	\[\vert(\varphi_1\ast\varphi_2)(t)-(\varphi_1\ast\varphi_2)(t')\vert
	\le\int_\R\vert\varphi_1(t-s)-\varphi_1(t'-s)\vert\cdot\vert\varphi_2(s)\vert\de s.\]
	Using the fact that $\varphi_2$ is bounded and the mapping $\R\to L^1(\R)$, $t\mapsto \varphi_1(t-\,\cdot)$ is continuous,
	we thus obtain $\varphi_1\ast\varphi_2\in\cC(\R,\C)$.
	From (ii) we further derive that
	$\cA_{a_1} * \cA_{a_2} \subeq \cA_{a_1 + a_2}$, so that 
	$\cA$ is a $*$-subalgebra of $L^1(\R)$.
	We also note that, for $\phi \in \cA_a$, 
	\[ \Vert\phi\Vert_1 \leq p_a(\phi) \int_\R \ee^{-t^2/a}\, dt
	= p_a(\phi) \sqrt{\pi a},\]
	so that the inclusions $(\cA_a, p_a) \to L^1$ are continuous. 
	
	\nin 
	\eqref{lem:gaussians_item4}
	For $b > a$, we have $p_b(\gamma_{a,z}) < \infty$
	for all $z \in \C$. It follows in particular that
	$\gamma_{a,z} \in \cA_b \subeq \cA$.
	That the map $F \colon \C \to \cA_b, F(z) = \gamma_{a,z}$ is holomorphic
	with respect to the norm $p_b$ follows by its local boundedness
	and the holomorphy of all functions
	$z \mapsto F(z)(x)$, $x \in \R$
	(\cite[Cor.~A.III.3]{Ne99}).
\end{proof}

\begin{defn}
	\label{def:expgro}
	Let $\cY$ be a 
	sequentially complete, 
	Hausdorff, locally convex space 
	and   $(U_t)_{t\in\R}$  a one-parameter group
	in $\GL(\cY)$ satisfying the following conditions: 
	\begin{enumerate}
		\item\label{agrowth_hyp1} 
		For every $y\in\cY$, the orbit map $U^y\colon \R\to\cY$, $U^y(t):=U_ty$, is continuous. 
		\item\label{agrowth_hyp2}  
		There exists a set
		of  seminorms $\cP$ defining the topology of $\cY$ such that 
		\begin{equation}\label{agrowth_eq1}
(\forall p\in\cP)\, (\exists q\in\cP,  B,C > 0)\, (\forall t\in\R)\, (\forall y
 \in \cY)\quad p(U_ty)\le C\ee^{B\vert t\vert}q(y). 
		\end{equation}
	\end{enumerate}
	We then say that $(U,\cY)$ has {\it exponential growth}
	or is a {\it one-parameter group of exponential growth}. 
\end{defn}

\begin{lem} \label{lem:comp-cont}
	\begin{footnote}
		{We thank the referee for this useful observation that simplified
			our assumptions considerably.}
	\end{footnote}
	If $(U,\cY)$ has exponential growth,
	then $\lim_{t \to 0} U_t = \1$ holds uniformly on compact subsets of $\cY$.
\end{lem}

\begin{proof} From \eqref{agrowth_eq1} we derive that, for each $r > 0$, the
  subset 
\[ \Gamma_r := \{ U_t \colon |t| \leq r \}\subeq \cL(\cY)\] is equicontinuous. 
	Hence the assertion follows from  
	the fact that the topology of pointwise convergence and compact convergence coincide on equicontinuous subsets of $\cL(\cY)$ (Remark~\ref{lintop}\eqref{lintop_item2}). 
\end{proof}

\begin{prop}\label{agrowth} {\rm(Gaussian operator calculus)}
	If $(U,\cY)$ has exponential growth, then the following assertions hold: 
	\begin{enumerate}[{\rm(i)}]
		\item\label{agrowth_item1} 
		For every $\varphi\in\cA$ there exists an
		operator $U(\varphi)\in\cL(\cY)$
		satisfying                 
\[(\forall\eta\in\cY^\sharp)(\forall y\in\cY)\quad 
\langle U(\varphi)y,\eta\rangle
=\int_\R\overline{\varphi(t)}\langle U_ty,\eta\rangle\de t 
\] 
		and the mapping
		$U\colon \cA\to\cL(\cY)$, $\varphi\mapsto  U(\varphi)$, 
		is  an algebra homomorphism. 
		\item\label{agrowth_item2} 
		For every $y\in \cY$ we have 
\[ 			\lim_{a\to 0}U(\gamma_{a,0})y=y \quad
			\mbox{ in} \quad \cY, \] 
		and, for $\eta\in \cY^\sharp$, we have 
\[ 			\lim_{a\to 0}U(\gamma_{a,0})^\sharp\eta=\eta
  \quad \mbox{ in } \quad \cY^\sharp.\]
Here the first relation holds uniformly on every compact subset of~$\cY$, while the second relation holds in~$\cY^\sharp_c$. 
		\item\label{agrowth_item3} 
		For $a > 0$ and $y\in\cY$,
		the function 
		$F_{a,y}\colon\C\to\cY, F_{a,y}(z):=U(\gamma_{a,z})y$
		is  holomorphic and 
		\[ F_{a,y}(s)=U_sU(\gamma_{a,0})y
		\quad \mbox{for all } s\in\R.\]
		Similarly, for $\eta\in\cY^\sharp$
		the function $F^\sharp_{a,\eta}\colon\C\to\cY^\sharp$, 
		$F^\sharp_{a,\eta}(z):=U(\gamma_{a,\oline{z}})^\sharp\eta$
		is weakly holomorphic and 
		\[ F^\sharp_{a,\eta}(s)=U^\sharp_s U(\gamma_{a,0})^\sharp\eta
		\quad \text{ for all }\quad s\in\R.\]
	\end{enumerate}
\end{prop}

\begin{proof} (i) The existence of the integral defining $U(\varphi)y\in\cY$
	and satisfying the asserted relation 
	follows from Lemma~\ref{impr}  since $\cY$ is sequentially complete.
	For $\varphi_1,\varphi_2\in\cA$, $y\in\cY$, and $\eta\in\cY^\sharp$ we have 
	\allowdisplaybreaks
	\begin{align*}
&		\langle U(\varphi_1)U(\varphi_2)y,\eta\rangle	
=\int_\R\overline{\varphi_1(t)}\langle
U_tU(\varphi_2)y,\eta\rangle\de t  =\int_\R\overline{\varphi_1(t)}\langle U(\varphi_2)y,U_t^\sharp\eta\rangle\de t \\
&=\int_\R\overline{\varphi_1(t)}\int_\R\overline{\varphi_2(s)}
                                                                                                                        \langle U_sy,U_t^\sharp\eta\rangle\de s\de t
                                                                                                                        =\iint_{\R^2}\overline{\varphi_1(t)\varphi_2(s)}\langle U_{t+s}y,\eta\rangle\de s\de t \\
  &=\int_\R\overline{(\varphi_1\ast\varphi_2)(t)}\langle U_ty,\eta\rangle\de t
 =	\langle U(\varphi_1\ast\varphi_2)y,\eta\rangle, 
	\end{align*}
	hence $U(\varphi_1)U(\varphi_2)=U(\varphi_1\ast\varphi_2)$. 
	The other properties are straightforward.
	
	To see that $U(\phi)$ defines a continuous operator on $\cY$,
	let $p$ be a continuous seminorm on $\cY$.
	Then there exist $B, C > 0$ and another continuous seminorm $q$ on
	$\cY$ such that $p(U_t v)\leq C \ee^{B|t|} q(v)$ for all $t \in \R$.
	This leads to 
	\[ p(U(\phi)v) \leq
	\int_\R |\phi(t)| p(U_t v)\, \de t 
	\leq    C \underbrace{\int_\R \ee^{B|t|} |\phi(t)| \, \de t}_{< \infty} \cdot q(v).\]
	We conclude that $U(\phi)$ is a continuous operator on $\cY$.
	
	\nin \eqref{agrowth_item2}
	Let $y\in\cY$. 
	For arbitrary $p\in\cP$ we select its corresponding $q,C,B$ as in~\eqref{agrowth_eq1}. 
	Using the equality 
	$1=\Vert\gamma_{a,0}\Vert_1 = \frac{1}{\sqrt{\pi a}}\int_\R  \ee^{-t^2/a}\de t$,
	we obtain  
	\begin{align*}
		U(\gamma_{a,0})y-y
		&=\frac{1}{\sqrt{\pi a}}\int_\R  \ee^{-t^2/a}(U_ty-y)\de t
		=\frac{1}{\sqrt\pi}
		\int_\R  \ee^{-t^2} (U_{\sqrt{a}t}y-y)\de t.
	\end{align*}
	Hence, for every $\delta>0$, we have by the estimate in Lemma~\ref{impr}(ii)
	\begin{align*}
		p(U(\gamma_{a,0})y-y)
		\le
		&
		\frac{1}{\sqrt\pi}
		\int_{[-\delta,\delta]}  \ee^{-t^2}p(U_{\sqrt{a}t}y-y)\de t \\
		&+
		\frac{2Cq(y)}{\sqrt\pi}
		\int_{\R\setminus[-\delta,\delta] }  \ee^{-t^2+B \vert t\vert}\de t.
	\end{align*}
	For arbitrary $\varepsilon>0$ let $\delta>0$ with 
	$\frac{2Cq(y)}{\sqrt\pi}
	\int_{\R\setminus[-\delta,\delta] } \ee^{-t^2+B \vert t\vert}\de t<\frac{\varepsilon}{2}$.
	Since
	\[ \lim_{s\to0}p(U_sy-y)=0,\]
	there exists $a_\varepsilon>0$ such that $\frac{1}{\sqrt\pi}
	\int_{[-\delta,\delta]}  \ee^{-t^2}p(U_{\sqrt{a}t}y-y)\de t<\frac{\varepsilon}{2}$   for all ${a\in(0,a_\varepsilon)}$. 
	Thus
	\[ p(U(\gamma_{a,0})y-y)<\varepsilon \quad \mbox{ if } \quad
	a\in(0,a_\varepsilon),\] 
	and this proves that $\lim_{a\to 0}U(\gamma_{a,0})y=y$. 
	
	For every $\eta\in\cY^\sharp$ and $y\in\cY$ we further obtain 
\[ \lim_{a\to 0}\langle y,U(\gamma_{a,0})^\sharp\eta\rangle
	=\lim_{a\to 0}\langle U(\gamma_{a,0})y,\eta\rangle
	=\langle y,\eta\rangle \]  
	hence $\lim_{a\to 0}U(\gamma_{a,0})^\sharp\eta=\eta$ in~$\cY^\sharp$. 
	
	The second part of the assertion follows along the same lines as above. 
	
	\nin	\eqref{agrowth_item3}
	For every $z\in\C$ and $\eta\in\cY^\sharp$ we have 
	\begin{align}\label{aFscalar}
		\langle F_{a,y}(z),\eta\rangle
		& =\int_\R\overline{\gamma_{a,z}(t)}
		\langle U_ty,\eta\rangle\de t
		=\int_\R\gamma_{a,\oline{z}}(t)
		\langle U_ty,\eta\rangle\de t \\
		& =\frac{1}{\sqrt{\pi a}}\int_\R\ee^{-(t-\oline{z})^2/a}
		\langle U_ty,\eta\rangle\de t 
		\nonumber
	\end{align}
	and it is straightforward to check that $\overline{\langle F_{a,y}(\cdot),\eta\rangle}\in\cO(\C)$. 
	Therefore $F_{a,y}\in\cO^w(\C,\cY)
        = \cO(\C,\cY)$ (Remark~\ref{rem:a.1}). 
	Moreover, if $s\in\R$, we have by the above equalities
 \allowdisplaybreaks
	\begin{align}
		\langle F_{a,y}(s),\eta\rangle 
		&=\frac{1}{\sqrt{\pi a}}\int_\R  \ee^{-(t-s)^2/a}
		\langle U_ty,\eta\rangle\de t 
=\frac{1}{\sqrt{\pi a}}\int_\R  \ee^{-t^2/a}
		\langle U_{t+s}y,\eta\rangle\de t  		\nonumber \\
		&\label{aUsUt}
           =\frac{1}{\sqrt{\pi a}}\int_\R  \ee^{-t^2/a}\langle U_sU_ty,\eta\rangle\de t \\
          &=\frac{1}{\sqrt{\pi a}}\int_\R  \ee^{-t^2/a}\langle U_ty,U^\sharp_s\eta\rangle\de t =\langle F_{a,y}(0),U^\sharp_s\eta\rangle 
			=\langle U_sF_{a,y}(0),\eta\rangle. \nonumber
	\end{align}
	Thus $U_sF_{a,y}(0)=F_{a,y}(s)$ for all $s\in\R$. 
	This shows that the weakly holomorphic $\cY$-valued function $F_{a,y}$ 
	is an extension of the orbit map $s\mapsto U_sF_{a,y}(0)$. 
	
	In order to obtain the antidual assertion, we note that,
	for every $\eta\in\cY^\sharp$, $y\in\cY$, and $z\in\C$,    
	\begin{equation}\label{aFs}
		\langle y,F^\sharp_{a,\eta}(z)\rangle
		=\langle y,U(\gamma_{a,\oline{z}})^\sharp\eta\rangle
		=\langle U(\gamma_{a,\oline{z}})y,\eta\rangle
		=\langle F_{a,y}(\oline{z}),\eta\rangle,
	\end{equation}
	hence $\langle y, F^\sharp_{a,\eta}(\cdot)\rangle  
	\in\cO(\C)$ 
	by~\eqref{aFscalar}. 
	Taking into account the identification
	$(\cY^\sharp)^\sharp=\cY$ via the antiduality pairing
	$\langle\cdot,\cdot\rangle$, it then follows that
	$F^\sharp_{a,\eta}\in \cO^w(\C,\cY^\sharp)$.  
	
	Finally, for every $s\in\R$ and $y\in\cY$, we have 
	\begin{align*}
		\langle y,U^\sharp_s (F^\sharp_{a,\eta}(0))\rangle
		&=\langle U_sy,F^\sharp_{a,\eta}(0)\rangle \mathop{=}^{\eqref{aFs}}
		\langle F_{a,U_sy}(0),\eta\rangle \\
		&\mathop{=}^{\eqref{aFscalar}}
		\frac{1}{\sqrt{\pi a}}\int_\R  \ee^{-t^2/a}\langle U_tU_sy,\eta\rangle\de t \mathop{=}^{\eqref{aUsUt}}	
		\langle F_{a,y}(s),\eta\rangle \\
&		\mathop{=}^{\eqref{aFs}}
		\langle y, F^\sharp_{a,\eta}(s)\rangle, 
	\end{align*}
	hence $F^\sharp_{a,\eta}(s)=U^\sharp_s (F^\sharp_{a,\eta}(0))$, 
	and the proof is complete. 
\end{proof}

As Proposition~\ref{agrowth}\eqref{agrowth_item3} 
ensures that the range of the operator $U(\gamma_{a,0})$ consists
of analytic vectors, we obtain from \eqref{agrowth_item2}
the  following corollary. 
Here we call a subset of a topological space  
$X\subeq Y$ {\it sequentially dense} if each point in
$Y$ is the limit of a sequence in $X$.

\begin{cor}
	\label{agrowth_cor}
	If $(U,\cY)$ has exponential growth, then 
	the spaces of weakly
        entire vectors for the one-parameter groups $(U_t)_{t\in\R}$ and $(U^\sharp_t)_{t\in\R}$ are sequentially dense in the spaces $\cY$ and $\cY^\sharp$, respectively. 
	Further, the space of weakly entire vectors for 
	$(U^\sharp_t)_{t\in\R}$ is
	sequentially dense in the space  $\cY^\sharp_c$. 
\end{cor}

\begin{ex}
	\label{ex1}
	We consider the Fr\'echet space $\cY := C(\R,\C)$
	of continuous functions on $\R$ and the one-parameter group defined by
	\[ (U_t f)(x) := e^{tx} f(x) \quad \mbox{ for } \quad t,x \in \R.\]
	The topology on $\cY$ is defined by the seminorms
	\[ p_n(f) := \sup \{ |f(x)| \colon |x| \leq n \}, \quad n \in \N.  \]
	They satisfy
	\[ p_n(U_t f) \leq e^{n|t|} p_n(f) \quad \mbox{ for } \quad t \in \R, n \in \N, f \in \cY\]
	and the constant $e^{n|t|}$ is sharp. 
	Therefore the constant $B$ in \eqref{agrowth_eq1} depends in
	general on $n$. 
\end{ex}

\begin{rem}\label{F-ex}
	For every $z\in\C$, $y\in\cY$, 
	and $a\in(0,\infty)$, 
	we have $U(\gamma_{a,0})y\in\cD(U_z)$ and $U_zU(\gamma_{a,0})y=U(\gamma_{a,z})y$ by Proposition~\ref{agrowth}\eqref{agrowth_item3}. 
	In particular $U_z$ is densely defined.
\end{rem}

\begin{cor}
	\label{F-tech_cor1}
	Suppose that $(U,\cY)$ has exponential growth. 
	Then we have 
	\[U(\gamma_{a,0})U_zy=U_zU(\gamma_{a,0})y=U(\gamma_{a,z})y
	\quad \mbox{ for } \quad a > 0, z \in \C, y\in\cD(U_z).
	\] 
\end{cor}

\begin{proof}
	The second equality in the statement follows for every $y\in\cY$
	by Remark~\ref{F-ex}.
	
	In order to prove that the leftmost and the rightmost expressions are equal, 
	we fix $y\in \cD(U_z)$ and 
	we recall that, if $w\in \oline{\cS}_{0,\Im z}$,
	then we have $\cD(U_z)\subseteq \cD(U_w)$, 
	hence it makes sense to define 
	$$G_{a,y}\colon \oline{\cS}_{0,\Im z}\to\cY,\quad G_{a,y}(w):=U(\gamma_{a,0})U_wy.$$
	As $U(\gamma_{a,0})$ is a
	continuous linear operator on $\cY$ (Proposition~\ref{agrowth}(i)) 
	and the function $w\mapsto U_wy$ belongs to $\cO_\partial(\oline{\cS}_{0,\Im z},\cY)$, 
	we have $G_{a,y}\in\cO_\partial(\oline{\cS}_{0,\Im z},\cY)$. 
	On the other hand, the function 
	$F_{a,y}\colon \C\to\cY$,  $F_{a,y}(w):=U(\gamma_{a,w})y$, 
	is holomorphic (Proposition~\ref{agrowth}\eqref{agrowth_item3}) 
	and satisfies $G_{a,y}=F_{a,y}$ on $\R$, hence $G_{a,y}=F_{a,y}$ on $\oline{\cS}_{0,\Im z}$. 
	In particular $G_{a,y}(z)=F_{a,y}(z)$, and this completes the proof. 
\end{proof}

\begin{cor}
	\label{F-tech_cor2}
	Assume that $(U,\cY)$ has exponential growth
	and write 
	\[\cD_0:=\{U(\gamma_{a,0})y:a > 0,  y\in\cY\}.\]
	If $z\in\C$ and the operator $U_z$ is closed, then $U_z$ is the closure of its restriction $U_z\vert_{\cD_0}\colon\cD_0\to\cY$. 
\end{cor}

\begin{proof}
	For all $y\in\cD(U_z)$ we have
	$y=\lim_{a\to 0}U(\gamma_{a,0})y$ and
	\[ U_zy=\lim_{a\to 0}U(\gamma_{a,0})U_zy
	=\lim_{a\to 0}U_zU(\gamma_{a,0})y \]
	by Proposition~\ref{agrowth}\eqref{agrowth_item2} 
	and Corollary~\ref{F-tech_cor1}. 
	Thus the operator $U_z$ is contained in the closure of its restriction $U_z\vert_{\cD_0}\colon\cD_0\to\cY$. 
	Since $U_z$ is a closed operator by hypothesis, 
	the assertion follows. 
\end{proof}

\begin{rem}
	\label{moderate}
	We recall from \cite[\S~2.1]{BK14} that a \emph{scale} on a Lie group $G$ is a function ${s\colon G\to(0,\infty)}$ such that both $s(\cdot)$ and $s(\cdot)^{-1}$ are locally bounded functions and moreover $s(gh)\le s(g)s(h)$ holds
	for every $g,h\in G$. 
	A representation $(\pi,E)$ of $G$ on a Fr\'echet space $E$ is a group homomorphism $\pi\colon G\to\GL(E)$ whose corresponding action mapping $G\times E\to E$ is jointly continuous. 
	The representation $(\pi,E)$ is called an \emph{F-representation} if there exists a sequence of seminorms $(p_n)_{n\in\N}$ defining the topology of $E$, such that for every $n\in \N$ the seminorm $p_n$ satisfies the following conditions: 
	\begin{enumerate}[{\rm(a)}]
		\item\label{moderate_a} 
		$p_n$ is \emph{$G$-continuous}, i.e, the action mapping $G\times (E,p_n)\to (E,p_n)$ is continuous; 
		\item\label{moderate_b} 
		$p_n$ is \emph{$s$-bounded}, i.e., there exist $C_n\in(0,\infty)$ and $N_n\in\N$ with 
\[ p_n \circ \pi(g)\le C_n s(g)^{N_n} p_n \quad \mbox{ for all } \quad g\in G.\]
	\end{enumerate}
	On the other hand, a representation $(\pi,E)$ is
	said to have \emph{moderate growth} if, for every continuous seminorm
	$p$ on $E$, there exist an integer $N\ge 1$ and a continuous seminorm $q$ on $E$ satisfying $p(\pi(g)v)\le s(g)^Nq(v)$ for every $v\in E$. 
	
	With this terminology, a representation
	$(\pi,E)$ of the Lie group $G$ (with a fixed scale $s$)  on the Fr\'echet
	space~$E$ is an F-representation if and only if $(\pi,E)$ has moderate growth, by \cite[Lemma~2.10]{BK14}. 
	
	In the special case of the Lie group $G=(\R,+)$ with its maximal
	scale, defined by $s(t):=\ee^t$ for every $t\in\R$ 
	(cf.~\cite[2.1.1]{BK14}), 
	it is easily seen that the moderate growth condition in the above sense for a one-parameter group $(U_t)_{t\in\R}$ on a
	Fr\'echet space $\cY$ is equivalent to the growth condition 
	in Definition~\ref{def:expgro}. 
	In this setting, Hypothesis~\ref{agrowth_hyp2} in
	Definition~\ref{def:expgro} implies
	that the action $\R\times \cY\to\cY$, $(t,y)\mapsto U_ty$, is continuous (\cite[Prop. 5.1]{Ne10}). 
	Consequently, if $\cY$ is a Fr\'echet space
        in Proposition~\ref{agrowth}, 
	then \cite[Lemma 2.10]{BK14} shows that the
	one-parameter group $(U_t)_{t\in\R}$ is an {$F$-representation}. 
	In particular, there exists a sequence of seminorms $(p_n)_{n\in\N}$ defining the topology of~$\cY$ for which the above condition~\eqref{moderate_a} is satisfied. 
\end{rem}

\begin{rem}
	\label{top}
	If $\cY$ is a locally convex space and $(U_t)_{t\in\R}$ is a one-parameter group in $\GL(\cY)$ satisfying the condition \eqref{agrowth_eq1} in
	Definition~\ref{def:expgro} 
	for some family of seminorms~$\cP$ defining the topology of~$\cY$, then
	\eqref{agrowth_eq1} is satisfied with~$\cP$ replaced by any other family of seminorms $\cP$ defining the topology of~$\cY$. 
	Thus the growth condition~\eqref{agrowth_eq1} depends only on the topology of $\cY$ rather than on the particular family of seminorms~$\cP$ involved in its statement. 
\end{rem}

\begin{rem}
	\label{loc}
	Let $\cY$ be a locally convex space  with a one-parameter group $(U_t)_{t\in\R}$, 
	in $\GL(\cY)$ and fix a family of seminorms $\cP$ defining the topology of~$\cY$. 
	We denote by $E(\cY)$ the set of all vectors $y\in\cY$ satisfying the 
	following conditions: 
	\begin{enumerate}
\item[\rm(E1)]\label{loc_item1} The orbit map $U^y\colon \R\to\cY$, $U^y(t):=U_ty$, is continuous. 
\item[\rm(E2)]\label{loc_item2} For every $p\in\cP$ we have 
\[ 
M_p(y):=\limsup_{\vert t\vert\to\infty}\frac{\log p(U_ty)}{\vert t\vert}<\infty.
\] 
\end{enumerate}
	Just as in Remark~\ref{top}, it is easily seen that the validity of 
	the above condition~(E2) does not depend on the particular choice of the 
	family of seminorms~$\cP$ defining the topology of $\cY$. 
	(More precisely, if $p_1,p_2$ are continuous seminorms on $\cY$ and $p_1\le cp_2$ for some constant $c>0$, then $M_{p_1}(y)\le M_{p_2}(y)$.)
	
	For every $y\in \cY$ and $p\in\cP$ we have 
	\begin{align*}
		M_p(y)<\infty 
		&\iff(\exists B>0)\ M_p(y)<B \\
	& \iff (\exists \delta>0)(\forall t\in\R,\ \vert t\vert>\delta)\ p(U_ty)\le \ee^{B\vert t\vert}.
\end{align*}
	If moreover the orbit map $U^y\colon \R\to\cY$ is continuous, then the function $p\circ U^y$ is continuous, hence bounded on the interval $[-\delta,\delta]$. 
	This shows that $E(\cY)$ is the set of all vectors $y\in\cY$ satisfying the above codition~(E1) and the condition 
	\begin{equation}
		\label{loc_item2prime}
		(\exists B,C>0)(\forall t\in\R)\ p(U_ty)\le C\ee^{B\vert t\vert}.
	\end{equation}
	Moreover, 
	\begin{equation}
		\label{Mp}
		M_p(y)=\inf\{B>0:(\exists C>0)(\forall t\in\R) \ p(U_ty)\le C\ee^{B\vert t\vert}\}.
	\end{equation}
	
	If $y_j\in E(\cY)$ and $B_j,C_j>0$ with $p(U_ty_j)\le C_j\ee^{B_j\vert t\vert}$ for all $t\in\R$, then 
	\[(\forall t\in\R)\ 
	p(U_t(y_1+y_2))\le C_1\ee^{B_1\vert t\vert}+C_2\ee^{B_2\vert t\vert}
	\le (C_1+C_2)\ee^{(B_1+B_2)\vert t\vert}.\]
	Therefore $y_1+y_2\in E(\cY)$ and moreover $M_p(y_1+y_2)\le B_1+B_2$ by \eqref{Mp}. 
	Taking the infimum with respect to $B_1$ and $B_2$ in this last inequality, we obtain  
	\[M_p(y_1+y_2)\le M_p(y_1)+M_p(y_2).\]
	In addition, it is easily checked that 
	\[(\forall w\in\C)\ M_p(wy)=M_p(y).\]
	In particular, $E(\cY)$ is a linear subspace of $\cY$. 
	
	As a by-product of the proof of Proposition~\ref{agrowth}, 
	if $\cY$ is Hausdorff and sequentially complete, 
	one obtains the following facts for every $y\in E(\cY)$: 
	\begin{itemize}
        \item[\rm(OC1)]
          There exists a linear mapping $U^y\colon\cA\to\cY$ satisfying 
\[ 			(\forall\varphi\in\cA)(\forall\eta\in\cY^\sharp)\quad 
			\langle U^y(\varphi),\eta\rangle
			=\int_\R\overline{\varphi(t)}\langle U_ty,\eta\rangle\de t.\] 		\item[\rm(OC2)] We have 
\[ 			\lim\limits_{a\to 0}U^y(\gamma_{a,0})=y \quad
			\mbox{ in} \quad \cY. \] 
		\item[\rm(OC3)] 	For every $a > 0$, 
		the function 
		$F_{a,y}\colon\C\to\cY, F_{a,y}(z):=U^y(\gamma_{a,z})$
		is weakly holomorphic and 
		\[ F_{a,y}(s)=U_s(U^y(\gamma_{a,0}))
		\quad \mbox{for all } s\in\R.\]
	\end{itemize}
	Thus $E(\cY)$ is contained in the closure of the space $\cY^\cO$
        of weakly entire vectors. 
	In particular, if $\cY^\cO=\{0\}$ then $E(\cY)=\{0\}$.  
\end{rem}

\begin{ex}(Banach spaces)
	\label{ex_B}
	In the notation of Remark~\ref{loc}, let us assume that $\cY$ is a Banach space and that $p(\cdot)=\Vert\cdot\Vert$ is a norm defining the topology of $\cY$. 
	If $(U_t)_{t\in\R}$ satisfies the condition~(E1) for every $y\in\cY$, 
	then $\limsup\limits_{\vert t\vert\to\infty}\frac{\log \Vert U_t\Vert}{\vert t\vert}<\infty$ by \cite[Eq. (10.2.2)]{HP57}, 
	hence $E(\cY)=\cY$ and moreover $(U,\cY)$
	is of exponential growth. 
\end{ex}

\begin{ex}
	\label{ex2}
	We consider the Fr\'echet space $\cY := C(\R,\C)$ (as in Example~\ref{ex1})
	and we define the one-parameter group 
	\[ (U_t f)(x) :=  f(x+t) \quad \mbox{ for } \quad t,x \in \R.\]
	It is clear that, for every $f\in\cY$, its orbit map $U^f\colon\R\to\cY$, $U^f(t):=f(\cdot+t)$ is continuous. 
	We now show by an example that $E(\cY)\subsetneqq\cY$. 
	
	In fact, let $f_0\colon\R\to\C$ satisfy 
	\begin{equation}
		\label{ex2_eq1}
		\limsup\limits_{\vert t\vert\to\infty}\frac{\log\vert f_0(t)\vert}{\vert t\vert}=\infty
	\end{equation}
	(for example $f_0(t)=\ee^{t^2}$ for every $t\in\R$). 
	We show that $f_0$ does not satisfy condition~(E2), 
	hence $f_0\not\in E(\cY)$. 
	We have already noted that condition~(E2) does not depend on the particular family of continuous seminorms~$\cP$, hence it suffices to show that there exists a continuous seminorm $p_0$ with ${M_{p_0}(f)=\infty}$. 
	To this end we consider the continuous seminorm 
	\[p_0\colon\cY\to[0,\infty),\quad p_0(f):=\vert f(0)\vert.\] 
	Then, for every $t\in\R$ and $f\in\cY$, we have $p_0(U_tf)=\vert f(t)\vert$, and therefore 
        $M_{p_0}(f_0)=\infty$ by \eqref{ex2_eq1}.  
	
	We now show that the space of
        entire vectors with respect to  $(U_t)_{t\in\R}$ is given by 
	\begin{equation}
		\label{ex2_eq2}
		\cY^\cO=\{v\in\cY : (\exists \widetilde{v}\in\cO(\C))\quad \widetilde{v}\vert_\R=v\}.
	\end{equation}
        The weakly holomorphic $\cY$-valued functions are
        actually holomorphic by  Remark~\ref{rem:a.1}
        since $\cY$ is a Fr\'echet space. 
To prove \eqref{ex2_eq2} we first note that 
	if $v\in\cY$, then $v\in\cY^\cO$ if and only if the orbit map $U^v\colon \R\to\cY$ has a holomorphic extension  $U^v\colon\C\to\cY$,  
	hence for every continuous linear functional $\xi\colon \cY\to\C$ the function $\xi\circ U^v\colon\C\to\C$ is holomorphic. 
	On the other hand, the evaluation functional $\xi_0\colon\cY\to\C$, $\xi_0(f):=f(0)$, 
	satisfies 
	\begin{equation}
		\label{ex2_eq3}
		(\forall f\in\cY)\quad \xi_0\circ U^f=f
	\end{equation}
	hence the function $\widetilde{v}:=\xi_0\circ U^v\colon\C\to\C$ satisfies 
	$\widetilde{v}\in\cO(\C)$ and $\widetilde{v}\vert_\R=v$. 
	
	Conversely, let $v$ belong to the right-hand side of~\eqref{ex2_eq2} 
	and define the function
        \[\widetilde{\varphi}\colon \C\to\cO(\C), \quad
          \varphi(z):=\widetilde{v}(z+\cdot),\] 
	and $\varphi\colon\C\to\cY$, $\varphi(z):=\widetilde{\varphi}(z)\vert_\R$, 
	that is, $\varphi(z)\colon\R\to\C$, $\varphi(z)(s):=\widetilde{v}(z+s)$. 
	Since $\widetilde{v}\vert_\R=v$, we have $\varphi(t)=U^v(t)$ for every $t\in\R$. 
	Moreover, for every $z,z_0\in\C$ we have 
	\[\lim_{z\to0}\frac{\widetilde{\varphi}(z)-\widetilde{\varphi}(z_0)}{z}=\widetilde{v}'(z_0+\cdot)\text{ in }\cO(\C)\]
where $\cO(\C)$ is considered as a Fr\'echet space with the
topology of uniform convergence on compact subsets of~$\C$. 
Since the mapping
\[ \C\to\cO(\C), \quad w\mapsto \widetilde{v}'(w+\cdot) \] 
is continuous, 	it follows that $\widetilde{\varphi}\colon \C\to\cO(\C)$ is a $C^1_\C$-map, hence holomorphic 
	(cf.~\cite[Lemma 2.1.1 and Prop. 2.1.6]{GN}). 
	Since the restriction map $R\colon \cO(\C)\to C(\R,\C)=\cY$ is linear and continuous, $R\circ \widetilde{\varphi}\colon \C\to\cY$ is a holomorphic mapping. 
	But $R\circ \widetilde{\varphi}=\varphi$, hence the extension $\varphi$ of the orbit map $U^v$ is holomorphic, and then $v\in\cY^\cO$. 
\end{ex}

\section{Duality theory for holomorphic extensions}
\label{Sect4}

This section contains our main results on holomorphic extensions of one-parameter groups of exponential growth 
on a suitable class of locally convex spaces, including all Fr\'echet 
spaces 
(the Duality Theorem \ref{UtoV} and Proposition~\ref{VtoU}). 
Our Duality Theorem asserts that, if $(U_t)_{t \in \R}$ is a one-parameter group
of exponential growth on 
a suitable space $\cY$ and 
$(V_t=U_t^\sharp)_{t\in\R}$ in $\GL(\cY^\sharp_c)$ is its antidual
one-parameter group,   then their respective holomorphic extensions satisfy
$V_{\oline{z}}= U_z^\sharp$ for every $z\in\C$. 
In order to address this result, we first prove a preliminary
closedness property of the holomorphic extension operators (Theorem~\ref{F-tech}) that
is of independent interest.

In the following, unless otherwise specified, $\cY$ stands for a Hausdorff locally convex space over~$\C$. 
We will need both the topological linear isomorphism 
\begin{equation}
	\label{biantid}
	\gamma\colon\cY_w\to(\cY^\sharp)^\sharp=\cY^{\sharp\sharp},\quad  
	\gamma(y):=\overline{\langle y,\cdot\rangle}
\end{equation}
and the continuous injective linear map given by the same expression, but taking values in a larger space
\begin{equation}
	\label{biantidcomp}
	\gamma_c\colon\cY\to(\cY^\sharp_c)^\sharp,\quad  
	\gamma_c(y):=\overline{\langle y,\cdot\rangle}.
\end{equation}
These maps are related by the equality 
\[\gamma_c=\id_c^\sharp\circ\gamma\circ \id_w\] 
in which the  identity maps
\[ \id_w\colon\cY\to\cY_w \quad \mbox{ and } \quad
  \id_c\colon \cY^\sharp_c\to \cY^\sharp \]
are continuous and bijective, but not homeomorphisms in general, 
hence the injective antidual map
$\id_c^\sharp\colon (\cY^\sharp)^\sharp\to (\cY^\sharp_c)^\sharp$ may
not be surjective. 

\begin{rem}
	\label{surjev}
	The mappings $\id_w$ and $\gamma$ are bijective, hence 
	$\gamma_c$ is surjective if and only if $\id_c^\sharp$ is surjective. 
	For later use, we note the following facts on surjectivity of $\gamma_c$. 
	\begin{enumerate}[{\rm(i)}]
		\item\label{surjev_item1} 
		The map $\gamma_c$ is surjective if and only if the closed convex hull of every compact subset of~$\cY$ is weakly compact. 
		In fact, using the well-known relation between $\cY^\sharp$ and the space of $\R$-linear continuous functionals on $\cY$ with values in~$\R$
		(cf.~Remark~\ref{annih}), 
		this can be derived from the analogous assertion for real locally convex spaces proved in \cite[Prop.~(15.1)]{Ba91}
		
		It then follows  
		that, if $\cY$ is quasi-complete
		(i.e., every bounded Cauchy net is convergent),
		then $\gamma_c\colon\cY\to(\cY^\sharp_c)^\sharp$ is surjective, hence bijective. 
		This result is stated in \cite[Cor. 8.17]{Au99} for complete spaces, 
		however, the method of proof is directly applicable for all quasi-complete locally convex  spaces since the closed convex hull of every compact subset is again compact in every quasi-complete space
		(see \cite[20.6.3, p.~241]{Ko69}).

		Thus, if $\cY$ is a Fr\'echet space, then $\gamma_c$ is bijective, 
		and it is actually a homeomorphism by \cite[Prop. (15.2)]{Ba91}. 
		\item\label{surjev_item2} 
		If $\gamma_c$ is surjective, then 
		any weakly $\cY^\sharp$-valued holomorphic map is weakly holomorphic
		as a $\cY^\sharp_c$-valued map. 
		More specifically, for every complex manifold $Z$ and every function $f\colon Z\to\cY^\sharp$ the following conditions are equivalent: 
		\begin{itemize}
			\item  $f\colon Z\to\cY^\sharp$ is weakly holomorphic on $Z$; 
			\item for every $y\in\cY$ the scalar function $\langle y,f(\cdot)\rangle$ is holomorphic on $Z$; 
			\item  $f\colon Z\to\cY^\sharp_c$ is weakly holomorphic on $Z$. 
		\end{itemize}
	\end{enumerate}
	
\end{rem}

The following lemma also holds without
the   exponential growth requirement.

\begin{lem}
	\label{easy}
	Assume that $\cY$ is a Hausdorff locally convex space over~$\C$ and $(U_t)_{t\in\R}$ is a one-parameter group in $\GL(\cY)$ with its corresponding antidual one-parameter group $(V_t=U_t^\sharp)_{t\in\R}$ in $\GL(\cY^\sharp_c)$. 
	Then for every $z\in\C$ we have 
	$V_{\oline{z}}\subseteq U_z^\sharp$ and $\gamma\circ U_z\circ\gamma^{-1}\vert_{\gamma(\cD(U_z))}\subseteq V_{\oline{z}}^\sharp$. 
\end{lem}

\begin{proof}
	Let $z=a+\ie b\in\C$ and assume $b>0$. 
	(The case $b<0$ can be similarly studied.)
	Also let $y\in\cD(U_z)\subseteq\cY$ and  $\eta\in\cD(V_{\oline{z}})\subseteq \cY^\sharp_c$ with their corresponding functions 
	$F_y\in \cO_\partial^w(\overline{\cS}_{0,b},\cY)$
	and $F_\eta\in \cO_\partial^w(\overline{\cS}_{0,-b},\cY^\sharp_c)$. 
	The function $F_\eta^\sim\colon\overline{\cS}_{0,b}\to\cY^\sharp_c$, $F_\eta^\sim(w):=F_\eta(\oline{w})$,
	is continuous on the closed strip~$\overline{\cS}_{0,b}$  and antiholomorphic on the open strip $\cS_{0,b}$. 
	Therefore
\[ \overline{\langle F_y(\cdot),  \eta\rangle},
	\overline{\langle y,F_\eta^\sim(\cdot)\rangle}\in\cO_\partial(\overline{\cS}_{0,b},\C) \]  
	and 
	\[(\forall s\in\R)\quad 
	\langle F_y(s),  \eta\rangle
	=\langle U_sy,  \eta\rangle
	=\langle y,  U_s^\sharp\eta\rangle
	=\langle y,  F_\eta(s)\rangle
	=\langle y,  F_\eta^\sim(s)\rangle.\]
	Hence
	$\langle F_y(z),  \eta\rangle=\langle y,  F_\eta^\sim(z)\rangle$
	for $z \in \cS_{0,b}$, 
	which leads to
	$\langle U_zy,  \eta\rangle=\langle y,  V_{\oline{z}}\eta\rangle$. 
	
	Thus, for every $\eta\in\cD(V_{\oline{z}})$, we have $\eta\in\cD(U_z^\sharp)$ and $U_z^\sharp\eta=V_{\oline{z}}\eta$. 
	On the other hand, for $y\in\cD(U_z)$ we have $\gamma(y)\in\cD(V_{\oline{z}}^\sharp)$ 
	and $\gamma(U_zy)=V_{\oline{z}}^\sharp\gamma(y)$. 
\end{proof}

We will work from now on in this section only with one-parameter groups satisfying the exponential growth condition
from Definition~\ref{def:expgro}.

\begin{rem}
	\label{F-equi1_rem}
	We note for later use that, for every 
	equicontinuous subset $E\subseteq \cY^\sharp$,  
	there exist $q\in\cP$  and  $M,B\in(0,\infty)$ such that 
	\begin{equation}
		\label{F-equi1}
		(\forall t\in\R)(\forall\eta\in E)(\forall y\in\cY)\quad 
		\vert\langle U_ty,\eta\rangle\vert\le M\ee^{B\vert t\vert}q(y).
	\end{equation}
	In fact, 
	since the subset $E\subseteq \cY^\sharp$ is equicontinuous, 
	there exist $p\in\cP$ and $M_1\in(0,\infty)$ such that for all $\eta\in E$ and $y\in\cY$ we have 
	$\vert\langle y,\eta\rangle\vert\le M_1 p(y)$. 
	Then, for  $t\in\R$, $\eta\in E$, and $y\in\cY$, we obtain 
	by~\eqref{agrowth_eq1}
	for suitable $q\in\cP$ and $B,C>0$:
	\begin{equation*}
		\vert\langle U_ty,\eta\rangle\vert
		\le M_1 p(U_ty) 
		\le M_1C\ee^{B\vert t\vert}q(y).
	\end{equation*}
\end{rem}

Using a more elaborate approach, one can show that the map from the first assertion of the following lemma is actually continuous,
cf.~Example~\ref{App_cont2}. 

\begin{lem}\label{cont}
	For every one-parameter group $(U,\cY)$ of exponential growth, 
	the following assertions hold: 
	\begin{enumerate}[{\rm(i)}]
\item\label{cont_item1} 
  For every compact subset $K\subseteq \cY$, the map
\[ \R\times K\to\cY, \quad (t,y)\mapsto U_ty,\] is continuous. 
\item\label{cont_item2} 
For every equicontinuous subset $E\subseteq \cY^\sharp$, the  mapping $\R\times E\to\cY^\sharp_c$, $(t,\eta)\mapsto U^\sharp_t\eta$, is continuous. 
	\end{enumerate}
\end{lem}

\begin{proof}
	\eqref{cont_item1}
	For every compact subset $K\subseteq\cY$, $y,y_0\in K$, and $t,t_0\in\R$ 
	we have 
	\[U_ty-U_{t_0}y_0=U_{t_0}(y-y_0)+(U_{t-t_0}-\1)U_{t_0}y.
	\] 
	Hence for every $p\in\cP$ 
	\[p(U_ty-U_{t_0}y_0)\le p(U_{t_0}(y-y_0))+\sup_{z\in U_{t_0}(K)}p((U_{t-t_0}-\1)z).\]
	Since the operator $U_{t_0}\colon\cY\to\cY$ is continuous, it follows that the subset $U_{t_0}(K)\subseteq\cY$ is compact.
	It then follows by Lemma~\ref{lem:comp-cont} that 
	the mapping $\R\times K\to\cY$, $(t,y)\mapsto U_ty$, is continuous at $(t_0,y_0)\in\R\times K$. 
	
	\nin	\eqref{cont_item2}
	We first recall from Remark~\ref{lintop}\eqref{lintop_item2}
	that the topologies of pointwise convergence and compact convergence coincide on the equicontinuous set $E\subseteq \cY^\sharp_c$. 
	For every  $\eta,\eta_0\in E$, and $t,t_0\in\R$ 
	we have as above 
	\[U^\sharp_t\eta-U^\sharp_{t_0}\eta_0
	=U^\sharp_{t_0}(\eta-\eta_0)+(U^\sharp_t\eta-U^\sharp_{t_0}\eta).
	\] 
	It is straightforward to check that the operator $U^\sharp_{t_0}\colon\cY^\sharp_c\to\cY^\sharp_c$ is continuous, hence, by the above equality, 
	it remains to prove that $\lim\limits_{t\to t_0}U^\sharp_t\eta=U^\sharp_{t_0}\eta$ uniformly for $\eta\in E$. 
	
	Since $E\subseteq \cY^\sharp$ is equicontinuous, there exist 
	$M>0$ and $p\in\cP$ with $\vert\langle y,\eta\rangle\vert\le M p(y)$ for all $y\in\cY$ and $\eta\in E$. 
	Then, for every compact subset $K\subseteq\cY$, we have 
	\begin{align*}
		\sup_{\eta\in E}\sup_{y\in K}\vert\langle y,U^\sharp_t\eta-U^\sharp_{t_0}\eta\rangle\vert
	 &=\sup_{\eta\in E}\sup_{y\in K}\vert\langle (U_{t-t_0}-\1)U_{t_0}y,\eta\rangle\vert \\
	&\le M\sup_{z\in U_{t_0}(K)}p(U_{t-t_0}z-z).
\end{align*}
	As above, the subset $U_{t_0}(K)\subseteq\cY$ is compact, 
	hence Lemma~\ref{lem:comp-cont}  implies 
	\[\lim\limits_{t\to t_0}\sup\limits_{\eta\in E}\sup\limits_{y\in K}\vert\langle y,U^\sharp_t\eta-U^\sharp_{t_0}\eta\rangle\vert=0.\] 
	Since the compact subset $K\subseteq\cY$ is arbitrary, 
	this completes the proof of the fact that the mapping
	$\R\times E\to\cY^\sharp_c$, $(t,\eta)\mapsto U^\sharp_t\eta$, is continuous at $(t_0,\eta_0)\in\R\times E$. 
\end{proof}

\begin{lem}
	\label{F-comp_orbit}
	For a continuous one-parameter group
	$(U,\cY)$ of exponential growth, 
	the following assertions hold for every compact subset $L\subseteq \C$: 
	\begin{enumerate}[{\rm(i)}]
		\item\label{F-comp_orbit_item1}
		For every 
		compact subset $K\subseteq\cY$ the set
                \[ \{\ee^{-(w+t)^2}U_ty: w\in L, t\in\R, y\in K\} \]
                is relatively compact in~$\cY$. 
		\item\label{F-comp_orbit_item2}
		For every equicontinuous 
		subset $E\subseteq\cY^\sharp_c$ 
		the set
\[ \{\ee^{-(w+t)^2}U^\sharp_t\eta: w\in L, t\in\R,\eta\in E\} \] is 
		equicontinuous in~$\cY^\sharp$. 
	\end{enumerate}
\end{lem}

\begin{proof} 
	\eqref{F-comp_orbit_item1} 
	It follows by \eqref{agrowth_eq1} in Definition~\ref{def:expgro}
	that 
	\[\lim\limits_{\vert t\vert\to\infty}\sup\limits_{w\in L}\sup\limits_{y\in K}p(\ee^{-(w+t)^2}U_ty)=0\] 
	for every seminorm $p\in\cP$, hence $\lim\limits_{\vert t\vert\to\infty}\ee^{-(w+t)^2}U_ty=0$ in $\cY$ uniformly for $w\in L$ and $y\in K$. 
	We now consider the one-point compactification of the real line $\R_\infty=\R \cup\{\infty\}$ and we define
	\[ f\colon \R_\infty\times L\times K\to\cY,
	 \quad
	f(t,w,y):=
	\begin{cases}
		\ee^{-(w+t)^2}U_ty & \text{ for } t\in\R \\ 
		0 & \text{ for } t = \infty.  \end{cases}
	\]
	It follows by Lemma~\ref{cont}\eqref{cont_item1} that $f\colon \R_\infty\times L\times K\to\cY$ is continuous, hence its image $f(\R_\infty\times L\times K)$ is a compact subset of~$\cY$. 
	In particular, $f(\R\times L\times K)$ is a relatively compact subset of $\cY$. 
	
	\nin	\eqref{F-comp_orbit_item2}
	The set $E$ is equicontinuous, hence we may use Remark~\ref{F-equi1_rem} to select $q\in\cP$  and  $M,B\in(0,\infty)$ with $\vert\langle y, U^\sharp_t\eta\rangle\vert\le M\ee^{B\vert t\vert}q(y)$ for all $t\in\R$, $\eta\in E$, and $y\in\cY$. 
	Since $\lim\limits_{\vert t\vert\to\infty}\ee^{-(w+t)^2}\ee^{B\vert t\vert}=0$, there exists $C\in(0,\infty)$ with $\vert \ee^{-(w+t)^2}\ee^{B\vert t\vert}\vert\le C$ for all $t\in\R$ and $w\in L$. 
	Therefore 
	\[
	\vert\langle y,  \ee^{-(w+t)^2}U^\sharp_t\eta\rangle\vert
	=\vert \ee^{-(w+t)^2}\vert \cdot \vert\langle y, U^\sharp_t\eta\rangle\vert
	\le M\ee^{B\vert t\vert}\vert \ee^{-(w+t)^2}\vert q(y)
	\le MC q(y).
	\]
	This estimate shows that the set  $\{\ee^{-(w+t)^2}U^\sharp_t\eta: w\in L, t\in\R,\eta\in E\}$ is 
	equicontinuous. 
\end{proof}

\begin{theorem}
	\label{UtoV} {\rm(Duality Theorem)} 
	If $(U,\cY)$ is of exponential growth and 
	$\gamma_c\colon\cY\to(\cY^\sharp_c)^\sharp$ is surjective 
	then
	\[ U_z^\sharp= V_{\oline{z}} \quad \mbox{  for every } \quad z\in\C.\] 
\end{theorem}

\begin{proof} We fix $z_0=a_0+\ie b_0\in\C$ and we assume $b_0>0$. 
	(The case $b_0<0$ can be discussed in the same manner.)
	We note that $\cD(U_{z_0})=\cD(U_{\ie b_0})$ and $U_{z_0}=U_{a_0}U_{\ie b_0}\subseteq U_{\ie b_0}U_{a_0}$ by 
	Proposition~\ref{holomext_prop}\eqref{holomext_prop_item2}, 
	which directly implies 
	\begin{equation}
		U_{a_0}^\sharp(\cD(U_{z_0}^\sharp))=\cD(U_{\ie b_0}^\sharp)
		\text{ and }U_{z_0}^\sharp=U_{\ie b_0}^\sharp U_{a_0}^\sharp.
	\end{equation}
	Similarly, 
	$\cD(V_{\oline{z}_0})=\cD(V_{-\ie b_0})$ 
	and $V_{\oline{z}_0}=V_{a_0}V_{-\ie b_0}=U_{a_0}^\sharp V_{-\ie b_0}$ by 
	Proposition~\ref{holomext_prop}\eqref{holomext_prop_item2}. 
	Hence we have
	\begin{align*}
		\cD(U_{z_0}^\sharp)\subseteq \cD(V_{\oline{z}_0})
		& \iff U_{a_0}^\sharp(\cD(U_{z_0}^\sharp))\subseteq U_{a_0}^\sharp(\cD(V_{\oline{z}_0})) \\
		& \iff  \cD(U_{\ie b_0}^\sharp)\subseteq  U_{a_0}^\sharp(\cD(V_{-\ie b_0})) \\
		& \iff  U_{-a_0}^\sharp(\cD(U_{\ie b_0}^\sharp))\subseteq  \cD(V_{-\ie b_0}) \\
		& \iff  \cD(U_{\ie b_0}^\sharp)\subseteq  \cD(V_{-\ie b_0}),  
	\end{align*}
	where the last equivalence follows from 
	$U_{-a_0}^\sharp(\cD(U_{\ie b_0}^\sharp))
	= \cD(U_{\ie b_0}^\sharp)$ by an application of \eqref{holomext_prop_proof_eq1} 
	for $w_0(=-a_0)\in\R$. 
	On the other hand, by Lemma~\ref{easy}, it suffices to prove $\cD(U_{z_0}^\sharp)\subseteq \cD(V_{\oline{z}_0})$. 
	Therefore, by the above equivalences, it will be enough to prove the
	inclusion: 
	\begin{equation}
		\label{red}
		\cD(U_{\ie b_0}^\sharp)
		\subseteq  \cD(V_{-\ie b_0}).
	\end{equation}
	To this end let $\eta\in \cD(U_{\ie b_0}^\sharp)$ be
	arbitrary, which is equivalent to the fact that there exist $M_1>0$ and $p_1\in\cP$ with  
	\begin{equation}
		\label{adj}
		(\forall y\in\cD(U_{\ie b_0}))\quad \vert \langle U_{\ie b_0}y,\eta\rangle\vert\le M_1p_1(y).  
	\end{equation}
	For every $z=a+\ie b\in\overline{\cS}_{-b_0,0}$ we have 
	$\cD(U_{\ie b_0})\subseteq\cD(U_{-\ie b})=\cD(U_{\oline{z}})$, 
	so that the following map is defined: 
	\begin{equation}
		\label{Psi-map}
\Psi\colon \cD(U_{\ie b_0})\to \cO_\partial(\overline{\cS}_{-b_0,0}),\quad 
 \Psi_y(z)	:=\langle U_{\oline{z}}y,
		\eta\rangle
		=\langle F_y({\oline{z}}),
		\eta\rangle. 
	\end{equation}
	\emph{We claim that $\Psi$ is continuous with respect to the topology of $\cY$.} 	
	In order to prove this claim, we define for every $y\in\cD(U_{\ie b_0})$ and $z=a+\ie b\in \overline{\cS}_{-b_0,0}$, 
	\begin{align*}
		f_y(z):= & \ee^{-z^2}\Psi_y(z)   \langle U_{\oline{z}}y,\ee^{-z^2}
	\eta\rangle  = \langle U_{-\ie b}y,\ee^{-(a+\ie b)^2}U_{a
	}^\sharp\eta\rangle \\
	=& \langle F_y(-\ie b),\ee^{-(a+\ie b)^2}U_{a
	}^\sharp\eta\rangle.
\end{align*}
	The subset 
	$\{\ee^{-(a+\ie b)^2}U_{a
	}^\sharp\eta:a\in\R,b\in[-b_0,0]\}\subseteq\cY^\sharp$ is equicontinuous
	by Lemma~\ref{F-comp_orbit}\eqref{F-comp_orbit_item2}, 
	hence there exist $M_2>0$ and $p_2\in\cP$ with 
	\begin{equation}\label{est}
		(\forall b\in[-b_0,0])\quad \vert f_y(a+\ie b)\vert\le M_2 p_2(F_y(-\ie b)).
	\end{equation}
	In particular, for all $a\in\R$ and $b\in[-b_0,0]$ we have 
	\[\vert f_y(a+\ie b)\vert\le M_2\sup \limits_{c\in[0,b_0]} p_2(F_y(\ie c)).\] 
	Since the function $F_y\colon\overline{\cS}_{0,b_0}\to\cY$ is
        continuous, the function $f_y\colon \overline{\cS}_{-b_0,0}\to\C$ is bounded. 
	We clearly have $f_y\in \cO_\partial(\overline{\cS}_{-b_0,0})$, hence, 
	by the  Phragm\'en--Lindel\"of Principle  
	(\cite[Th.~12.8]{Ru87}) 
	we obtain 
	\begin{equation}
		\label{bd}
		(\forall z\in \overline{\cS}_{-b_0,0})\quad 
		\vert f_y(z)\vert\le \max\{\sup_{s\in\R}\vert f_y(s)\vert,\sup_{s\in\R}\vert f_y(s-\ie b_0)\vert\}
	\end{equation}
	For every $s\in\R$ we have  by \eqref{est}
	\begin{equation}
		\label{bd0}
		\vert f_y(s)\vert
		\le M_2 p_2(F_y(0))=M_2 p_2(y)
	\end{equation}
	and 
	\begin{align*}		
		\vert f_y(s-\ie b_0)\vert
		&=\vert\ee^{-(s-\ie b_0)^2}\langle U_{s+\ie b_0}y,
		\eta\rangle\vert \\
		&=\vert\ee^{-(s-\ie b_0)^2}\langle U_{\ie b_0}U_{s
		}y,\eta\rangle\vert
		\qquad \qquad \mbox{ by } \eqref{adj} \\
		& \le \ee^{-s^2+b_0^2}M_1 p_1(U_{s
		}y) \\
		&\le M_1C\ee^{-s^2+b_0^2}\ee^{B\vert s
			\vert} p_3(y) 
		\qquad \qquad \mbox{ by } \eqref{agrowth_eq1}.  
	\end{align*}
	Here  we used \eqref{agrowth_eq1} for finding suitable constants $B,C>0$ and the seminorm  $p_3\in\cP$. 
	The above inequality shows that there exists a constant
	$M_3>0$ such that $\vert f_y(s-\ie b_0)\vert\le M_3p_3(y)$. 
	It follows by this inequality along with \eqref{bd} and \eqref{bd0} that 
	there exist a constant $M>0$ and a seminorm $p\in\cP$ with the property 
	\begin{equation*}
		(\forall z\in\overline{\cS}_{-b_0,0})(\forall y\in\cD(U_{\ie b_0}))\quad 
		\vert f_y(z)\vert\le Mp(y).
	\end{equation*}
	Since $f_y(z)=\ee^{-z^2}\Psi_y(z)$ for all $y\in\cD(U_{\ie b_0})$ and $z\in \overline{\cS}_{-b_0,0}$, the above estimate
	and its antilinearity directly imply that the mapping $y\mapsto \Psi_y$ in \eqref{Psi-map} is continuous, as claimed. 
	
	Recalling that $\cD(U_{\ie b_0})$ is dense in $\cY$ by 
	Corollary~\ref{agrowth_cor}, 
	we then extend $\Psi$ to a continuous linear mapping $\Psi\colon \cY\to \cO_\partial(\overline{\cS}_{-b_0,0})$ satisfying 
	\begin{equation}
		\label{Psi-est}
		(\forall z\in\overline{\cS}_{-b_0,0})(\forall y\in\cY)\quad 
		\vert \Psi_y(z)\vert\le M\vert \ee^{z^2}\vert p(y).
	\end{equation}
	For every $z\in\overline{\cS}_{-b_0,0}$ we now define 
	$\widetilde{\Psi}(z)\colon \cY\to\C$ by $\langle y,\widetilde{\Psi}(z)\rangle:=\Psi_y(z)$ 
	for all $y\in\cY$. 
	Then the functional $\widetilde{\Psi}(z)$ is antilinear by \eqref{Psi-map} 
	and is continuous by \eqref{Psi-est}, 
	hence ${\tilde\Psi(z)\in\cY^\sharp}$. 
	We thus obtain a 
	mapping $\widetilde{\Psi}\colon \overline{\cS}_{-b_0,0}\to\cY^\sharp$. 
	
	For every $y\in\cY$ we have  $\langle y,\widetilde{\Psi}(\cdot)\rangle:=\Psi_y(\cdot)\in\cO_\partial(\overline{\cS}_{-b_0,0})$, 
	hence
\[ \widetilde{\Psi}\in
	\cO^w_\partial(\overline{\cS}_{-b_0,0},\cY^\sharp).\] 
	On the other hand, it follows by \eqref{Psi-est} that for every compact subset $K\subseteq \overline{\cS}_{-b_0,0}$ the subset $\widetilde{\Psi}(K)\subseteq\cY^\sharp$ is equicontinuous. 
	Since the topology of compact convergence coincides with the weak-$*$-topology on every equicontinuous set
	(cf.~Remark~\ref{lintop}\eqref{lintop_item2}), 
	it follows that $\widetilde{\Psi}\in\cC(\overline{\cS}_{-b_0,0},\cY^\sharp_c)$. 
	Since 
	the mapping $\gamma_c\colon\cY\to(\cY^\sharp_c)^\sharp$ is surjective by hypothesis, we further
	have $\cO^w(\cS_{-b_0,0},\cY^\sharp)=\cO^w(\cS_{-b_0,0},\cY^\sharp_c)$
	(cf.~Remark~\ref{surjev}\eqref{surjev_item2})
	and therefore
	$\widetilde{\Psi}\in \cO^w_\partial(\overline{\cS}_{-b_0,0},\cY^\sharp_c)$. 
	
	Finally, for every $s\in\R$ and $y\in\cD(U_{\ie b_0})$ we have 
	\[ \langle y,\widetilde{\Psi}(s)\rangle=\Psi_y(s)=\langle U_sy,U_{-a_0}^\sharp\eta\rangle
	=\langle y,U_s^\sharp U_{-a_0}^\sharp\eta\rangle.\] 
	Since $\cD(U_{\ie b_0})$ is dense in $\cY$, this implies 
	$\widetilde{\Psi}(s)=U_s^\sharp U_{-a_0}^\sharp\eta$ for every $s\in\R$. 
	Thus $\widetilde{\Psi}\in \cO^w_\partial(\overline{\cS}_{-b_0,0},\cY^\sharp_c)$ 
	is an $\cO^w_\partial$-extension of the orbit map of $U_{-a_0}^\sharp\eta\in\cY^\sharp$, 
	hence $U_{-a_0}^\sharp\eta\in\cD(V_{-\ie b_0})$. 
	This completes the proof of \eqref{red}. 
\end{proof}

Next we introduce the extra structure required to deal with KMS
conditions.

\begin{rem}
	\label{annih}
	It is well known that the real duality pairing  \[\langle\cdot,\cdot\rangle_\R:=\Re\langle\cdot,\cdot\rangle\colon\cY\times \cY^\sharp\to\R\]
	defines the topological $\R$-linear isomorphism
	$\xi\mapsto \langle \cdot, \xi\rangle_\R$  
	from $\cY^\sharp$  (respectively, $\cY^\sharp_c$)
	onto the locally convex real vector space consisting of the $\R$-linear continuous functionals on $\cY$ endowed with the topology of pointwise convergence 
	(respectively, the topology of uniform convergence on compact subsets);
	see~\cite[\S 21.11]{Ko69}. 
	
	If $T\colon\cD(T)\subseteq\cY\to\cY$ is a densely-defined linear operator, then 
	\[T^\sharp\colon\cD(T^\sharp)\subseteq\cY^\sharp\to\cY^\sharp\] 
	is just the dual of the operator $T$ with respect to the duality pairing $\langle\cdot,\cdot\rangle_\R$. 
	
	If $J\colon\cY\to\cY$ is a continuous antilinear operator,  
	then there exists a unique continuous antilinear operator $J^\sharp\colon\cY^\sharp\to\cY^\sharp$
	satisfying 
	$\overline{\langle Jy,\eta\rangle}=\langle y, J^\sharp\eta\rangle$ 
	for all $y\in\cY$ and $\eta\in\cY^\sharp$. 
	Again,  $J^\sharp$ is the dual of the $\R$-linear operator $J$ with respect to the duality pairing $\langle\cdot,\cdot\rangle_\R$. 
	
	For any 
	$\eta\in\cD(T^\sharp)\subseteq \cY^\sharp$ 
	we have 
	\begin{align*}
		(T^\sharp -J^\sharp)\eta=0 
		&\iff (\forall y\in\cD(T)) \quad \langle y,	(T^\sharp -J^\sharp)\eta\rangle_\R=0\\
		&\iff (\forall y\in\cD(T)) \quad \langle (T-J)y,\eta\rangle_\R=0\\
		&\iff \langle (T-J)\cD(T),\eta\rangle_\R=\{0\} \\
		&\iff \Re\langle (T-J)\cD(T),\eta\rangle=\{0\}.
	\end{align*}
	This shows that the real linear subspace $\Ker(T^\sharp-J^\sharp)$ is closed in $\cY^\sharp$ 
	and can be regarded as the annihilator of the range of the operator $T-J$ with respect to the real duality pairing $\langle\cdot,\cdot\rangle_\R$. 
\end{rem}

\begin{cor}
	\label{UtoV_cor}
	If $(U,\cY)$ is of exponential growth, 
	$\gamma_c\colon\cY\to(\cY^\sharp_c)^\sharp$ is surjective, 
	and  $J\colon\cY\to\cY$ is a continuous antilinear operator, 
	then, for every $z\in\C$, the $\R$-linear subspace $\Ker(V_z-J^\sharp)\subseteq\cY^\sharp$ is closed and we have 
	\begin{equation}
		\label{UtoV_cor_eq1}
		\Ker(V_z-J^\sharp)=((U_{\bar{z}}-J)\cD(U_{\bar{z}}))^\perp, 
	\end{equation}
	where the annihilator is computed with respect to the real duality pairing 
	$\langle\cdot,\cdot\rangle_\R\colon\cY\times \cY^\sharp\to\R$. 
\end{cor}

\begin{proof}
	We have $U_{\bar{z}}^\sharp=V_z$ by the Duality Theorem~\ref{UtoV}, hence Remark~\ref{annih} implies \eqref{UtoV_cor_eq1}, which in turn implies that $\Ker(V_z-J^\sharp)\subseteq\cY^\sharp$ is closed. 
\end{proof}

We now prove the following complement to the Duality Theorem~\ref{UtoV}.

\begin{prop}
	\label{VtoU}
	If $(U,\cY)$ is of exponential growth and 
	$\gamma_c\colon\cY\to(\cY^\sharp_c)^\sharp$ is surjective, 
	then we have  $\gamma\circ U_z\circ\gamma^{-1}\vert_{\gamma(\cD(U_z))}
	= V_{\oline{z}}^\sharp$ for every $z\in\C$. 
\end{prop}

The proof of this fact requires some preparations,
including Theorem~\ref{F-tech} below.

\begin{rem}
	\label{closable}
	We say that a linear operator $T\colon\cD(T)\to\cY$ is \emph{closable} if it has an extension which is a closed linear operator. 
	If this is the case and $\cD(T)$ is dense in $\cY$, then $\cD(T^\sharp)$ dense in $\cY^\sharp$ 
	and moreover $T^\sharp$ is a closed operator by Remark~\ref{adjgraph}. 
	Iterating this property, we obtain the densely-defined closed linear operator $T^{\sharp\sharp}:=(T^\sharp)^\sharp\colon\cD(T^{\sharp\sharp})\to\cY^{\sharp\sharp}$.  
	Recalling from \eqref{biantid} the topological linear isomorphism $\gamma\colon\cY_w\to(\cY^\sharp)^\sharp=\cY^{\sharp\sharp}$,  
	it is straightforward to check that $\gamma(\cD(T))\subseteq\cD(T^{\sharp\sharp})$ and 
	$\gamma\circ T= T^{\sharp\sharp}\circ\gamma\vert_{\cD(T)}$. 
	Moreover, the densely-defined closed linear operator
	\[ \gamma^{-1}\circ T^{\sharp\sharp}\circ\gamma\vert_{\gamma^{-1}(\cD(T^{\sharp\sharp}))}\colon \gamma^{-1}(\cD(T^{\sharp\sharp}))\to\cY \]
	is the closure of the operator $T$ (\cite[Ch. IV, 7.1]{Sch71}). 
\end{rem}

\begin{theorem}
	\label{F-tech} {\rm(Closed Operator Theorem)}
	If $(U_t)_{t \in \R}$ is a one-parameter group of
	exponential growth 
	and $\gamma_c\colon\cY\to(\cY^\sharp_c)^\sharp$ is surjective,  
	then
	for every $z\in\C$, the operator $U_z\colon\cD(U_z)\to\cY$ is linear, densely defined, and closed. 
\end{theorem}

Recall from Remark~\ref{surjev}(i) that $\gamma_c$ is surjective if 
$\cY$ is a Fr\'echet space, and, more generally, if $\cY$ is quasi-complete. 

\begin{proof} Let $z_0=a_0+\ie b_0\in\C$.  
	It is clear from its definition that $\cD(U_{z_0})$ is a $\C$-linear subspace of $\cY$ and the operator $U_{z_0}\colon \cD(U_{z_0})\to\cY$ is linear. 
	It follows by Proposition~\ref{agrowth} that $\cD(U_{z_0})$ is dense
	in~$\cY$. 
	
	To prove that $U_{z_0}$ is a closed operator in $\cY$, let $(x_\iota)_{\iota\in I}$ be a net in $\cD(U_{z_0})$ for which there exist $x,y\in \cY$ with $\lim\limits_{\iota\in I}x_\iota =x$ and $\lim\limits_{\iota\in I}U_{z_0}x_\iota =y$ in $\cY$. 
	From now on we assume $b_0>0$.  
	The case $b_0<0$ then follows by replacing $U_t$ by $U_{-t}$.
	Since $x_\iota\in\cD(U_{z_0})$, we have its corresponding  function $F_{x_\iota}\in \cO^w_\partial(\overline{\cS}_{0,b_0},\cY)$ 
	with $F_{x_\iota}(t)=U_tx_\iota$ for all $t\in\R$. 
	If $w=a+\ie b\in \overline{\cS}_{0,b_0}$, then
	\begin{equation*}
		F_{x_\iota}(w)=F_{x_\iota}(a+\ie b)=U_aU_{\ie b}x_\iota
	\end{equation*} 
	by Proposition~\ref{holomext_prop}\eqref{holomext_prop_item2}. 
	For $\eta\in\cY^\sharp$ we then obtain 
	\begin{equation}
		\label{F-tech_prf_eq1}
		\overline{\langle \ee^{-w^2}F_{x_\iota}(w),\eta\rangle}
		=\ee^{-(a+\ie b)^2}\overline{\langle U_aU_{\ie b}x_\iota,\eta\rangle}.
	\end{equation}
	Hence, by \eqref{F-equi1},
	there exist $M,B\in(0,\infty)$ and a continuous seminorm
	$q\colon\cY\to\R_+$, such that   
	\begin{align*}
		\vert\langle \ee^{-w^2}F_{x_\iota}(w),\eta\rangle\vert 
		&
           =\ee^{-a^2+b^2}\vert\langle U_a U_{\ie b}x_\iota,\eta\rangle\vert
           \le \ee^{-a^2+b^2} M \ee^{B\vert a\vert}q(U_{\ie b}x_\iota)\\
&           =\ee^{-a^2+b^2}  \ee^{B\vert a\vert}M_\iota, 
	\end{align*}
	where the constant $M_\iota:=M\sup\limits_{b\in[0,b_0]}q(U_{\ie b}x_\iota)$ is finite 
	since the function
	$[0,b_0]\to\R$, $b\mapsto q(U_{\ie b}x_\iota)
	=q \circ F_{x_\iota}(\ie b)$ is continuous
	on the compact interval~$[0,b_0]$. 
	The above estimate shows that, for every $\eta\in\cY^\sharp$, the
	function 
	$$f_{\iota,\eta}\in\cO_\partial(\overline{\cS}_{0,b_0},\C), 
	\quad 
	f_{\iota,\eta}(w):=\overline{\langle \ee^{-w^2}F_{x_\iota}(w),\eta\rangle}$$ 
	is bounded. 
	
	We also obtain by \eqref{F-tech_prf_eq1}
	\begin{equation}
		\label{F-tech_prf_eq2}
		(\forall a\in\R)\quad 
		f_{\iota,\eta}(a)=\ee^{-a^2}\overline{\langle U_ax_\iota,\eta\rangle}
		=\overline{\langle x_\iota,\ee^{-a^2}U^\sharp_a\eta\rangle}.
	\end{equation}
	For every 
	equicontinuous subset 
	$E\subseteq\cY^\sharp$,  
	the set $\{\ee^{-a^2}U_a\eta: a\in\R,\eta\in E\}$ is 
	equicontinuous in $\cY^\sharp$ 
	by Lemma~\ref{F-comp_orbit}\eqref{F-comp_orbit_item2}, 
	and $\lim\limits_{\iota\in I}x_\iota =x$ in~$\cY$.
	Hence we obtain by Remark~\ref{lintop}\eqref{lintop_item1}
	\begin{equation}
		\label{F-tech_prf_eq3}
		\lim\limits_{\iota\in I}f_{\iota,\eta}(a)
		=\oline{\langle x,\ee^{-a^2}U^\sharp_a\eta\rangle} 
		\quad \text{\ \ \  uniformly for }a\in\R\text{ and }\eta\in E. 
	\end{equation}
	We similarly obtain for $a, a_0\in\R$ and $z_0 = a_0 + \ie b_0$
	by \eqref{F-tech_prf_eq1}
	\allowdisplaybreaks
	\begin{align*}
		f_{\iota,\eta}(a+ib_0)
          & =\ee^{-(a+\ie b_0)^2}\overline{\langle U_aU_{\ie b_0}x_\iota,\eta\rangle}
  =\ee^{-(a+\ie b_0)^2}\overline{\langle U_{a-a_0}U_{a_0}U_{\ie b_0}x_\iota,\eta\rangle} \\
&=\ee^{-(a+\ie b_0)^2}\overline{\langle U_{a-a_0}U_{a_0+\ie b_0}x_\iota,\eta\rangle} 
=\ee^{-(a+\ie b_0)^2}\overline{\langle U_{a-a_0}U_{z_0}x_\iota,\eta\rangle} \\
		&=\ee^{-(a+\ie b_0)^2}\overline{\langle U_{z_0}x_\iota,U^\sharp_{a-a_0}\eta\rangle} =\overline{\langle U_{z_0}x_\iota,\ee^{-(a-\ie b_0)^2}U^\sharp_{a-a_0}\eta\rangle}.
	\end{align*}
	For every 
	equicontinuous subset 
	$E\subseteq\cY^\sharp$ the set
	\[ \{\ee^{-(a-\ie b_0)^2}U^\sharp_{a-a_0}\eta:a\in\R\}=
	\{\ee^{-(a+a_0-\ie b_0)^2}U^\sharp_a\eta:a\in\R,\eta\in E\}\] 
	is equicontinuous in $\cY^\sharp$ 
	by Lemma~\ref{F-comp_orbit}\eqref{F-comp_orbit_item2}  again. Hence  
	\[ \lim\limits_{\iota\in I}U_{z_0}x_\iota =y \quad \mbox{ in } \quad
	\cY \]
	entails by Remark~\ref{lintop}\eqref{lintop_item1}
	\begin{equation}
		\label{F-tech_prf_eq4}
		\lim\limits_{\iota\in I}f_{\iota,\eta}(a+ib_0)
		=\ee^{-(a+\ie b_0)^2}\overline{\langle y,U^\sharp_{a-a_0}\eta\rangle}
		\text{ uniformly for }a\in\R,\eta\in E.
	\end{equation} 
	Thus 
	$(f_{\iota,\eta}\colon\overline{\cS}_{0,b_0}\to\C)_{\iota\in I}$
	is a net of bounded $\cO_\partial$-functions which is uniformly convergent on the boundary of the strip $\overline{\cS}_{0,b_0}$, 
	the corresponding convergence being also uniform with respect to $\eta\in E$ for  every equicontinuous subset 
	$E\subseteq\cY^\sharp$. 
	It then follows by the Phragm\'en--Lindel\"of Principle  
	(\cite[Th.~12.8]{Ru87})  that there exists a bounded $\cO_\partial$-function $f_\eta\colon\overline{\cS}_{0,b_0}\to\C$ with 
	$\lim\limits_{\iota\in I}f_{\iota,\eta}=f_\eta$ uniformly on 
	the strip~$\overline{\cS}_{0,b_0}$ and uniformly
	with respect to $\eta\in E$ for  every equicontinuous subset 
	$E\subseteq\cY^\sharp$. 
	This further implies that $f_\eta$ is a $\cO_\partial$-function and the mapping $\eta\mapsto f_\eta$ is antilinear since both these properties are preserved through uniform convergence on $\overline{\cS}_{0,b_0}$. 
	
	We now claim that, for every $w=a+\ie b\in\overline{\cS}_{0,b_0}$,  
	the functional
	\[f(w)\colon\cY^\sharp\to\C,\quad f(w)(\eta):=f_\eta(w)\]
	belongs to $\cY$, in the sense that there exists $\widetilde{f}(w)\in\cY$ with
        \[f(w)(\eta)=\overline{\langle \widetilde{f}(w),\eta\rangle}
          \quad \mbox{  for all } \quad \eta\in\cY^\sharp.\]  
	We already noted that the mapping 
	$\eta\mapsto f_\eta(w)	$ is antilinear, 
	hence the functional $f(w)$ is antilinear. 
	In order to prove that $f(w)$ is continuous on $\cY^\sharp$ with respect to its weak-$*$-topology, we first consider the case 
	when $w$ belongs to the boundary of the strip, that is, either $b=0$ or $b=b_0$. 
	If $b=0$, then $w=a\in\R$ and we have 
	by \eqref{F-tech_prf_eq2}--\eqref{F-tech_prf_eq3}
	\begin{equation}
		\label{F-tech_prf_eq5}
		f(a)(\eta)
		=\lim_{\iota\in I}\overline{\langle x_\iota,\ee^{-a^2}U^\sharp_a\eta\rangle}
		=\overline{\langle x,\ee^{-a^2}U^\sharp_a\eta\rangle}
		=\overline{\langle\ee^{-a^2}U_ax,\eta\rangle}, 
	\end{equation}
	hence we may take $\widetilde{f}(a):=\ee^{-a^2}U_ax$. 
	Moreover, 
	\[(\forall\eta\in\cY^\sharp)\quad 
	\vert f_\eta(a)\vert
	=\vert f(a)(\eta)\vert\le\sup\limits_{K_x}\vert\langle\cdot,\eta\rangle\vert\]
	where the subset $K_x:=\{\ee^{-a^2}U_a x:a\in\R\}\subseteq \cY$ is relatively compact by Lemma~\ref{F-comp_orbit}\eqref{F-comp_orbit_item1}. 
	Similarly, for $w=a+\ie b_0$, we have by \eqref{F-tech_prf_eq4}
	\begin{equation}
		\label{F-tech_prf_eq6}
		f(a+\ie b_0)(\eta)
		=\ee^{-(a+\ie b_0)^2}\overline{\langle y,U^\sharp_{a-a_0}\eta\rangle}
		=\overline{\langle \ee^{-(a+\ie b_0)^2}U_{a-a_0}y,\eta\rangle}
	\end{equation}
	hence  we may take $\widetilde{f}(a+\ie b_0):=\ee^{-(a+\ie b_0)^2}U_{a-a_0}y$. 
	We have 
	\[(\forall\eta\in\cY^\sharp)\quad 
	\vert f_\eta(a+\ie b_0)\vert
	=\vert f(a+\ie b_0)(\eta)\vert
	\le\sup\limits_{L_y}\vert\langle \cdot,\eta\rangle\vert\]
	where the subset $L_y:=\{\ee^{-(a+\ie b_0)^2}U_{a-a_0}y:a\in\R\}\subseteq \cY$ is relatively compact by Lemma~\ref{F-comp_orbit}\eqref{F-comp_orbit_item1} again. 
	
	Since $f_\eta\colon\overline{\cS}_{0,b_0}\to\C$ is a bounded $\cO_\partial$-function, 
	it then follows by the Phragm\'en--Lindel\"of Principle  
	(\cite[Th.~12.8]{Ru87}) that 
	\[(\forall w\in\overline{\cS}_{0,b_0})(\forall\eta\in\cY^\sharp) \quad 
	\vert f(w)(\eta)\vert
	\le\sup\limits_{K_x\cup L_y}\vert\langle \cdot,\eta\rangle\vert.
	\]
	This shows in particular that,  for every $w\in\overline{\cS}_{0,b_0}$,  
	the antilinear functional $f(w)\colon\cY^\sharp\to\C$ is continuous 
	if $\cY^\sharp$ is endowed with the topology of uniform convergence on the compact subsets of~$\cY$, that is, $f(w)\in(\cY^\sharp_c)^\sharp$. 
	Since 
	$\gamma_c\colon\cY\to(\cY^\sharp_c)^\sharp$
	is surjective by hypothesis, it then follows   
	that  there exists a unique element $\widetilde{f}(w)\in\cY$ with
        \[ f(w)(\eta)=\overline{\langle \widetilde{f}(w),\eta\rangle}
          \quad \mbox{  for all } \quad \eta\in\cY^\sharp.   \]

	Consequently, the function  $\widetilde{f}\colon \overline{\cS}_{0,b_0}\to\cY$  
	has the property that,  
	for every $\eta\in\cY^\sharp$, the function 
	\[ \overline{\cS}_{0,b_0}\to\C,  \quad
	w\mapsto\overline{\langle \widetilde{f}(w),\eta\rangle}=f(w)(\eta)=f_\eta(w)\]
	is a bounded $\cO_\partial$-function. 
	For every $\eta\in\cY^\sharp$ we have 
	\[\lim_{\iota\in I}\overline{\langle \ee^{-w^2}F_{x_\iota}(w),\eta \rangle}
	=\lim_{\iota\in I}f_{\iota,\eta}(w)=f(w)(\eta)=\overline{\langle\widetilde{f}(w),\eta\rangle}
	\]
	uniformly with respect to  $w\in\overline{\cS}_{0,b_0}$, and 
	to $\eta\in E$, for every 
	equicontinuous 
	subset $E\subseteq\cY^\sharp$. 
	Here $F_{x_\iota}\in \cO^w_\partial(\overline{\cS}_{0,b_0},\cY)$, 
	hence $F_{x_\iota}\colon\overline{\cS}_{0,b_0}\to\cY$ 
	is continuous when $\cY$ is endowed with the topology of uniform convergence with respect to
	equicontinuous  subsets of~$\cY^\sharp$, 
	which coincides with the topology of $\cY$ by Remark~\ref{lintop}\eqref{lintop_item1}. 
	It then follows from the above uniform limit 
	that 
	${\widetilde{f}\in \cO^w_\partial(\overline{\cS}_{0,b_0}, \cY)}$ as well. 
	
	Moreover, $\widetilde{f}(a)=\ee^{-a^2}U_ax$ for every $a\in\R$ by \eqref{F-tech_prf_eq5}. 
	This shows that the function
	\[ F_x\colon \overline{\cS}_{0,b_0}\to\C, \quad
	F_x(z):=\ee^{z^2}\widetilde{f}(z)\]
      is an $\cO_\partial$-extension of the orbit map $a\mapsto U_ax$, hence
      $x\in\cD(U_{z_0})$ and  
	${U_{z_0}x=F_x(z_0)}$. 
	We have for every $\eta\in\cY^\sharp$ 
	\[
	\langle F_x(z_0),\eta\rangle
	=\langle F_x(a_0+\ie b_0),\eta\rangle 
	=\langle \ee^{(a_0+\ie b_0)^2} \widetilde{f}(a_0+\ie b_0),\eta\rangle
	\mathop{=}\limits^{\eqref{F-tech_prf_eq6}}
	\langle y,\eta\rangle
	\]
	hence $F_x(z_0)=y$, that is, $U_{z_0}x=y$. 
	This  completes the proof 
	of the closedness of the operator $U_{z_0}\colon\cD(U_{z_0})\to\cY$. 
\end{proof}

\begin{proof}[Proof of Proposition~\ref{VtoU}]
	It follows by Theorem~\ref{F-tech} that the operator $U_z$ is closed. 
	Then, by  Remark~\ref{closable}, we have $U_z=\gamma^{-1}\circ U_z^{\sharp\sharp}\circ\gamma$. 
	By the Duality Theorem~\ref{UtoV} we have $U_z^{\sharp\sharp}=V_{\oline{z}}^\sharp$, and the assertion follows. 
\end{proof}

\section{One-parameter groups 	on spaces of smooth vectors}
\label{Sect5}

This brief section has a rather technical character and provides a bridge between the abstract results of Sections \ref{Sect3} and \ref{Sect4} and their applications
in Section~\ref{Sect6} in the framework of unitary representations of Lie groups.
The main point is to verify that, for a unitary representation
$(U,\cH)$ of a Lie group~$G$, the natural action
of a one-parameter group $U_{h,t} := U(\exp th)$ on the Fr\'echet space
$\cY := \cH^\infty$ of smooth vectors and its dual space
$\cY^\sharp = \cH^{-\infty}$, the distribution vectors, defines a 
one-parameter group of exponential growth.

\begin{hyp}
	\label{concr_hyp}
	In Sections~\ref{Sect5}--\ref{Sect6} we work in the following setting: 
	\begin{itemize}
		\item $G$ is a finite-dimensional real Lie group with its Lie algebra $\fg$.
		\item $(U,\cH)$ is a unitary representation of $G$. 
		\item $\cH^\infty:=\{\xi\in\cH:U^\xi\in \cC^\infty(G,\cH)\}$ is the space of smooth vectors endowed with its unique Fr\'echet topology for which the inclusion map $\cH^\infty\hookrightarrow\cH$ is continuous 
		(cf.~\cite[Cor.~1.1]{Go69}). To specify a set of defining seminorms,
		we fix a basis $\bx=(x_1,\dots,x_m)\in\fg^m$ in $\fg$ 
		and we denote $\cI_k:=\{1,\dots,m\}^k$ and  $\cI:=\bigsqcup\limits_{k\in\Z_+}\cI_k$. 
		For arbitrary $k\in\Z_+$ and $I=(i_1,\dots,i_k)\in\cI_k$ we define the seminorm
		$$p_I\colon\cH^\infty\to\R_+,\quad p_I(\xi):=\Vert \de U(x_{i_1})\cdots\de U(x_{i_k})\xi\Vert.$$
		Then the topology of $\cH^\infty$ is defined by the set of seminorms
                \[ \cP_0:=\{p_I : I \in \cI\}.\] 
		We also denote by $\cP$ the family of seminorms consisting of finite sums of elements of~$\cP_0$. 
		\item The space $\cH^{-\infty}$  	of continuous antilinear functionals on $\cH^\infty$
		is endowed with its weak-$*$-topology and we write
		\[ \langle\cdot,\cdot\rangle\colon\cH^\infty\times\cH^{-\infty}\to\C \]
		for the antiduality pairing that coincides on $\cH^\infty \times \cH$
		with the scalar product of~$\cH$ (antilinear in the first variable). 
\item $(U^\infty,\cH^\infty)$ is the restriction of the representation $(U,\cH)$ and for every $\xi\in\cH^\infty$ we have
		$U^{\infty,\xi}\in C^\infty(G,\cH^\infty)$ by
		\cite[Prop.~1.2]{Po72} or by  \cite[Thm.~4.4]{Ne10}. 
\item For $g\in G$ we define the continuous linear operator $U^{-\infty}(g)$ on $\cH^{-\infty}$ by 
		$$(\forall\xi\in\cH^\infty)(\forall\eta\in\cH^{-\infty})\quad \langle U^\infty(g^{-1})\xi,\eta\rangle=\langle \xi,U^{-\infty}(g)\eta\rangle$$
		and we thus get a separately continuous
		representation $(U^{-\infty},\cH^{-\infty})$ of $G$;
                see Appendix~\ref{AppB} for a discussion of
                its continuity properties. 
		\item 	We also define for $h\in\fg$ and $t \in \R$ 
		\begin{equation}
			\label{hyp_hyperb_eq1}
			U_{h,t}:=U(\exp_G(th)),
                        \ U^{\pm\infty}_{h,t}:=U^{\pm\infty}(\exp_G(th)).
		\end{equation}
	\end{itemize}
\end{hyp}

\begin{rem}
	\label{semin_gen}
	Let $(\cdot \mid \cdot)$ be the
	real scalar product on $\fg$  with $(x_i\mid x_j)=\delta_{ij}$ (Kronecker's delta) for all $i,j=1,\dots,m$.   
	We denote by $\Vert\cdot\Vert$ the corresponding norms of the linear operators on~$\fg$.
	For $i,j\in\{1,\dots,m\}$,  
	we define the functions 
	\begin{equation}
		\label{aij}
		a_{ij}\colon\g\to\R,\quad 
		a_{ij}(y):=(\Ad_G(\exp_G(-y))x_i\mid x_j)
	\end{equation}
	hence
	\begin{equation}
		\label{eigen_gen}
		(\forall y \in \g)\quad 
		\Ad_G(\exp_G(-y))x_i=\sum_{j=1}^ma_{ij}(y)x_j
	\end{equation}
	and	then
	\begin{align*}
		\vert a_{ij}(y)\vert
		&=\vert (\Ad_G(\exp_G(-y))x_i\mid x_j)
		=\vert (\ee^{-\ad y}x_i\mid x_j)\vert 
		\le \Vert \ee^{-\ad y}\Vert 
		\le \ee^{\Vert \ad y\Vert}. 
	\end{align*}
	
	On the other hand, for every $y \in \g$ and $\xi\in\cH^\infty$, we have 
	\begin{align}
		U(\exp_G(-y)) & (\de U(x_{i_1}) 
		 \cdots\de U(x_{i_k}))U(\exp_Gy)\xi \nonumber \\
		&=
		\de U(\Ad_G(\exp_G(-y))x_{i_1})\cdots 
		\de U(\Ad_G(\exp_G(-y))x_{i_k})\xi \nonumber \\
		&\label{equal_gen}
		=\sum_{j_1,\dots,j_k=1}^m 
		a_{i_1 j_1}(y)\cdots a_{i_k j_k}(y)
		\de U(x_{j_1})\cdots\de U(x_{j_k})\xi
	\end{align}
	where the last equality follows by \eqref{eigen_gen}. 
	Then  
	\begin{align*}
		\Vert\de U(x_{i_1}) 
		\cdots\de U(x_{i_k}) & U(\exp_Gy)\xi\Vert  \\
		\le 
		& \sum_{j_1,\dots,j_k=1}^m 
		\vert a_{i_1 j_1}(y)\cdots a_{i_k j_k}(y)\vert 
		\Vert \de U(x_{j_1})\cdots\de U(x_{j_k})\xi\Vert 
		\\
		\le 
		&\ee^{k \Vert \ad y\Vert}\sum_{j_1,\dots,j_k=1}^m 
		\Vert \de U(x_{j_1})\cdots\de U(x_{j_k})\xi\Vert.
	\end{align*}
	That is, 
	\begin{equation}
		\label{homog_gen}
		(\forall I\in\cI_k)	\quad 
		p_I(U(\exp_Gy)\xi)
		\le 
		\ee^{k \Vert\ad y\Vert}\sum_{J\in\cI_k}
		p_J(\xi).
	\end{equation}
	Also 
	\begin{align*}
			(\de U &(x_{i_1}) 
		\cdots \de U(x_{i_k}))(U(\exp_Gy)\xi -\xi) \\
		=&
		U(\exp_Gy) \de U(\Ad_G(\exp_G(-y))x_{i_1})\cdots 
		\de U(\Ad_G(\exp_G(-y))x_{i_k})\xi \\
		&-\de U(x_{i_1})\cdots\de U(x_{i_k})\xi \\
		=&\sum_{j_1,\dots,j_k=1}^m 
		a_{i_1 j_1}(y)\cdots a_{i_k j_k}(y)
		U(\exp_Gy)\de U(x_{j_1})\cdots\de U(x_{j_k})\xi \\
		& -\de U(x_{i_1})\cdots\de U(x_{i_k})\xi.
	\end{align*}
	Since the functions $a_{ij}\colon\g\to\R$ are continuous and satisfy
	$a_{ij}(0)=\delta_{ij}$ by~\eqref{aij}, 
	we then obtain 
	\begin{equation*}
		\lim_{y\to 0}\Vert (\de U(x_{i_1}) 
		\cdots\de U(x_{i_k}))(U(\exp_Gy)\xi -\xi)\Vert=0. 
	\end{equation*}
	We thus have 
	\begin{equation}\label{cont_gen}
		\lim_{y\to 0}
		p_I(U(\exp_Gy)\xi-\xi)=0 \quad \mbox{ for every } I \in \cI,
	\end{equation}
	which asserts the strong continuity of the
	$G$-action on $\cH^\infty$.
	This can also be derived from the fact that the $G$-action 
	on $\cH^\infty$ is smooth (\cite[Thm.~4.4]{Ne10}).  
\end{rem}

\section{Applications to standard subspaces} 
\label{Sect6}

In this section we eventually come to results
that have been the our main motivation to develop
an abstract theory of holomorphic extensions of one-parameter groups 
on locally convex spaces.
So let $(U,\cH)$ be a unitary representation of the
Lie group $G$ on $\cH$, $h \in \g$,
$U_{h,t} := U(\exp th)$ for $t \in\R$,
and $J \colon\cH \to \cH$ a conjugation preserving $\cH^\infty$
and commuting with each $U_{h,t}$. 

We then obtain a
standard subspace
\[ \sV := \Fix(J\Delta^{1/2}) \quad \mbox{ for }\quad
\Delta := e^{2 \pi \ie \partial U(h)}\]
(cf.~\cite[\S 3.1]{NO17}) and 
\cite[Prop.~2.1]{NOO21} implies that 
\begin{equation}
  \label{eq:v-hkms}
   \sV = \cH_{\rm KMS}.   
\end{equation}
Our main results can now be stated as follows: 
\begin{itemize}
	\item  $\cH^{-\infty}_{\rm KMS}$  is a (weak-$*$-)closed subspace of $\cH^{-\infty}$ (Theorem~\ref{extcl}); 
	\item  $\cH^{-\infty}_{\rm KMS}\cap\cH=\sV$ (Theorem~\ref{thm:KMSV}); 
	\item $\sV$ is dense in $\cH^{-\infty}_{\rm KMS}$ (Theorem~\ref{thm:vdense}). 
\end{itemize}

\begin{defn}
	\label{concr_conj}
	Since  the inclusion map $\cH^\infty\hookrightarrow\cH$ is continuous 
	and $\cH^\infty$~is a Fr\'echet space, it follows from the Closed Graph Theorem that the antilinear operator $J^\infty:=J\vert_{\cH^\infty}\colon\cH^\infty\to\cH^\infty$ is continuous. 
	Then, as in Remark~\ref{annih}, we can define the antilinear operator 
	\[J^{-\infty}:=(J^\infty)^\sharp\colon \cH^{-\infty}\to\cH^{-\infty}
	\quad \mbox{ 	by  } \quad
	J^{-\infty}\eta=\overline{\eta\circ J^\infty} \mbox{
		for every } \eta\in\cH^{-\infty}. \] 
	We denote by $\cH^{-\infty}_{\rm ext}$ the set of all $\eta\in\cH^{-\infty}$ whose corresponding orbit map 
	\[ U^{-\infty,\eta}_h\colon\R\to \cH^{-\infty}, \quad t\mapsto U^{-\infty}_{h,t}\eta,  \]
	has an $\cO_\partial^w$-extension
        to the strip 
$U^{-\infty,\eta}_h\colon\overline{\cS}_{0,\pi}\to\cH^{-\infty}$. 
That this actually is an $\cO_\partial$-extension
  follows from Remarks~\ref{rem:a.1} and \ref{rem:a.2}.
	We also define 
	\[\cH^{-\infty}_{\rm KMS} :=\{\eta\in\cH^{-\infty}_{\rm ext}: U^{-\infty,\eta}_h(\ie\pi)=J^{-\infty}\eta\}.\] 
\end{defn}

The following theorem is an important consequence
of the Duality Theorem~\ref{UtoV} and its corollary. It is one of the main
results of this paper.

\begin{theorem}
	\label{extcl}
	$\cH^{-\infty}_{\rm KMS}$ is a closed real linear subspace of
	the space $\cH^{-\infty}$ of distribution vectors, and we have 
	\begin{equation*}
		\cH^{-\infty}_{\rm KMS}=((U_{h,\ie\pi}^\infty-J^\infty)\cD(U_{h,\ie\pi}^\infty))^\perp
	\end{equation*}
	where the annihilator is computed with respect to the real duality pairing 
	$\langle\cdot,\cdot\rangle_\R\colon\cH^\infty\times \cH^{-\infty}\to\R$. 
\end{theorem}

\begin{proof}
	We specialize Proposition~\ref{agrowth} 
	as follows: 
	\[ 	\cY:=\cH^\infty,\quad 	\cY^\sharp:=\cH^{-\infty},\quad
	U_t:=U^\infty_{h,t},\quad V_t:=U^{-\infty}_{h,-t}.	\]
	Then 	$\cY$ 
	is a Fr\'echet space (cf.~Hypothesis~\ref{concr_hyp}). 
	Furthermore, the two requirements for exponential growth 
	follow by  \eqref{homog_gen} and \eqref{cont_gen}. 
	So the Duality Theorem~\ref{UtoV} applies and yields
	$(U_{h, \ie \pi}^\infty)^\sharp  = V_{-\ie \pi}$. 
	For $z:=-\ie\pi$ we have 
	\[\cD(V_z)=\cH^{-\infty}_{\rm ext}\text{ and } 
	\Ker(V_z-J^\sharp)
	=\cH^{-\infty}_{\rm KMS}, \] 
	so that the assertion follows from Corollary~\ref{UtoV_cor}. 
\end{proof}

The following lemma is a refinement of \cite[Prop.~2.1]{NOO21}. 

\begin{lem} \label{lem:weak-KMS}
	If $\xi \in \cH$ is such that the orbit map
	$U^\xi \colon \R \to \cH$ extends to a weakly continuous map
	on $\oline{\cS}_{0,\pi}$   that is weakly holomorphic on $\cS_{0,\pi}$,
	then $\xi \in \sV = \cH_{\rm KMS}$. 
\end{lem}

\begin{proof} 
	We consider the continuous positive-definite function 
	\[  
	\psi_\xi\colon\R\to\C,\quad 
	\psi_\xi(t)
	= \la \xi, U_{h,t} \xi \ra.  
	\]
	By assumption, the continuous function 
	$\psi_\xi$ extends to a
	continuous function on $\oline{\cS}_{0,\pi}$ which is holomorphic
	on the open strip. 
	Using the relation
	\[ 
	\psi_\xi(t + z)
	= \la \xi, U_{h,t} U^\xi(z) \ra
	= \la U_{h,-t} \xi, U^\xi(z) \ra
	\quad \mbox{ for }\quad t \in \R, \Im z \in [0,\pi], \]
	we even obtain by analytic continuation 
	\[ 
	\psi_\xi(z + w) 
	= \la U^\xi(-\oline w) , U^\xi(z) \ra
	\quad \mbox{ for }\quad z, w \in \oline{\cS}_{0,\pi}.\]
	This shows that 
	$\psi_\xi$ 
	actually extends to a
	continuous function on $\oline{\cS}_{0,2\pi}$, holomorphic on the interior.
	
	Now the kernel
	\[ K(z,w) := \la U^\xi(w), U^\xi(z) \ra
	=  
	\psi_\xi(z- \oline w) \] 
	on $\oline{\cS}_{0,\pi}$ is continuous and sesquiholomorphic
	on the interior. 
	Therefore the function
	$U^\xi \colon \oline{\cS}_{0,\pi} \to \cH$ is continuous
	and holomorphic on the interior (cf.~\cite[Lemma~A.III.11]{Ne99}). 
\end{proof}

The following theorem is used crucially in \cite{FNO23}
to construct nets of standard subspaces on 
non-compactly causal symmetric spaces. 

\begin{theorem} 
	\label{thm:KMSV}
	$\cH^{-\infty}_{\rm KMS} \cap \cH = \sV$.
\end{theorem}

\begin{proof} 
	By \eqref{eq:v-hkms}, we have
    $\sV = \cH_{\rm KMS} \subeq \cH^{-\infty}_{\rm KMS} \cap \cH$. It remains to show the converse inclusion.
	Let $\eta \in \cH^{-\infty}_{\rm KMS} \cap \cH$, i.e.,
	the orbit map of $\eta$ in $\cH^{-\infty}$ extends analytically 
        to a map
	\[ U^{-\infty,\eta}_h \colon \oline{\cS}_{0,\pi} \to \cH^{-\infty} \]
	that is weak-$*$-continuous and satisfies
	$U^{-\infty,\eta}_h(\pi \ie) = J \eta$.
	This implies in particular that
	\begin{equation}
		\label{eq:17}U^{-\infty,\eta}_h(\R \cup (\R + \pi \ie)) \subeq \cH.
	\end{equation}
	Let $\xi \in \cH^\infty$ be a smooth vector. 
	Then the function
	\[ f = f^\xi \colon  \oline{\cS}_{0,\pi} \to \C, \quad
	f^\xi(z) := \la \xi, U^{-\infty,\eta}_h(z) \ra \]
	is continuous and holomorphic on $\cS_{0,\pi}$. 
	We claim that
	\begin{equation}
		\label{eq:claim}
		|f(z)| \leq \Vert\xi\Vert \cdot \Vert\eta\Vert \quad \mbox{ for } \quad
		z \in  \oline{\cS_{0,\pi}}.
	\end{equation}
	By \eqref{eq:17}, this relation holds trivially for
	$z \in \partial \cS_{0,\pi}$.
	We want to derive our claim from the 
	Phragm\'en--Lindel\"of Principle   (\cite[Th.~12.8]{Ru87}),
	so  we have to show that $f$ is bounded.
	
	For the functions
	\[ \varphi_n:=\gamma_{1/n,0}\colon\R\to\R,\quad \phi_n(t) := \sqrt{\frac{n}{\pi}} e^{-n t^2}, \] 
	Lemma~\ref{lem:gaussians} and Proposition~\ref{agrowth}
	imply that,
	for $\eta \in \cH^{-\infty}$, we obtain a sequence
	\[ \eta_n = U_h^{-\infty}(\phi_n)\eta = \sqrt{\frac{n}{\pi}}
	\int_\R e^{-n t^2}\, U^{-\infty}_{h,t}\eta\, dt  \]
	of distribution vectors for which the orbit map
	$U^{-\infty, \eta_n} \colon \R \to \cH^{-\infty}$ has a holomorphic extension and
	that
	$\lim\limits_{n \to \infty} \eta_n = \eta$ in the weak-$*$-topology
	of $\cH^{-\infty}$.
	
	For $\eta \in \cH^{-\infty}_{\rm KMS} \cap \cH$, we then have
	$\eta_n \in \cH$, 
	and $\Vert\phi_n\Vert_1 = 1$ implies 
	\begin{equation*}
		\Vert\eta_n\Vert \leq \Vert\eta\Vert \quad \mbox{ for }  \quad n \in \N.
	\end{equation*}
	We consider the functions 
	\[
	f_n\colon\C\to\C,\quad 
	f_n(z) :=  \la \xi, U_h^{-\infty, \eta_n}(z) \ra=   \la \xi,
	U_h(\delta_z * \phi_n) \eta \ra,\]
	where the map
	\[ \C \to L^1(\R), \quad
	z \mapsto \delta_z * \phi_n=\gamma_{1/n,z}, 
	\]
	is holomorphic by Lemma~\ref{lem:gaussians}.
	Therefore $f_n$ is an entire function with
	\[ 
	(\forall z\in\C)\quad 
	|f_n(z)| \leq \Vert\xi\Vert \cdot \Vert\eta\Vert \cdot \Vert\delta_z * \phi_n\Vert_1.\]
	Since the right hand side is independent of $\Re z$ and locally bounded,
	the function $f_n$ is bounded on $\oline{\cS}_{0,\pi}$.
	Moreover, for $z =x\in \R$ we have
	\[ |f_n(x)| \leq \Vert\xi\Vert \cdot \Vert\eta\Vert, \]
	and for $z = x + \pi \ie$, we have
	\[ |f_n(x + \pi \ie )|
	= |\la \xi, U_h^{-\infty, \eta_n}(x + \pi i) \ra|
	= |\la \xi, J U_h^{-\infty, \eta_n}(x) \ra|
	\leq  \Vert\xi\Vert \cdot \Vert\eta\Vert.\]
	The Phragm\'en--Lindel\"of Principle (\cite[Th.~12.8]{Ru87})
	now implies that
	\[ |f_n(z)| \leq \Vert\xi\Vert \cdot \Vert\eta\Vert \quad \mbox{
		for }  \quad z \in \oline{\cS_{0,\pi}}.\]
	Next we observe that
	\[ f_n(z)
	= \la  \xi, U_h(\delta_z * \phi_n) \eta \ra
	= \la  \xi, U_h(\phi_n) U^{-\infty,\eta}_h(z) \ra
	\to    \la  \xi, U^{-\infty,\eta}_h(z) \ra = f(z) \]
	for each $z \in \oline{\cS}_{0,\pi}$.
	This implies our claim in \eqref{eq:claim}
	for each smooth vector $\xi \in \cH^\infty$. It then follows
	that the linear functionals $U^{-\infty,\eta}_h(z)$ on $\cH^\infty$ extend 
	continuously to $\cH$, which means that $U^{-\infty,\eta}_h(z) \in \cH$,
	as we consider $\cH$ as a subspace of $\cH^{-\infty}$.
	
	We thus obtain a map
	\[ U^{-\infty,\eta}_h \colon \oline{\cS}_{0,\pi} \to \cH \]
	extending the orbit map of $\eta$ which  satisfies
	$\Vert U^{-\infty,\eta}_h(z)\Vert \leq \Vert\eta\Vert$ for each $z \in \oline{\cS}_{0,\pi}$
	by \eqref{eq:claim}. 
	So $U^{-\infty,\eta}_h$ is bounded.
	Since the functions $f^\xi$ 
	are continuous for $\xi \in \cH^\infty$ and $\cH^\infty$ is dense in $\cH$,
	it follows that $U^{-\infty,\eta}_h$ is weakly continuous.
	On the open strip $\cS_{0,\pi}$ the functions
	$f^\xi$ are bounded and holomorphic for $\xi \in \cH^\infty$,
	so that \cite[Cor.~A.III.3]{Ne99} implies that
	$U^{-\infty,\eta}_h \colon \cS_{0,\pi} \to \cH$ is holomorphic.
	Finally Lemma~\ref{lem:weak-KMS} shows that
	$\eta \in \cH_{\rm KMS} = \sV$.
\end{proof}

\begin{theorem} \label{thm:vdense} 
	$\sV$ is dense in $\cH^{-\infty}_{\rm KMS}$.   
\end{theorem}

We shall need the following lemmas to prove Theorem~\ref{thm:vdense}:

\begin{lem} \label{lem:cartes} {\rm(\cite[Lemma~3.7]{NO21})} 
	Let $X$ be a locally compact space and 
	$f \colon X \to \cH^{-\infty}$ be a weak-$*$-continuous map. 
	Then the following assertions hold: 
	\begin{enumerate}[{\rm(a)}]
		\item
		$f^\wedge \colon X \times \cH^\infty \to \C, f^\wedge(x,\xi) := f(x)(\xi)$, 
		is continuous. 
		\item
		If, in addition, $X$ is a complex manifold and 
		$f$ is antiholomorphic, then $\oline{f^\wedge}$ is holomorphic. 
	\end{enumerate}
\end{lem}

\begin{lem} \label{lem:2.11} For each smooth vector
	$v \in \sV' \cap \cH^\infty$ and $\eta \in \cH^{-\infty}_{\rm KMS}$,
	we have $\eta(v) \in \R$.   
\end{lem}

\begin{proof} 
	We consider the one-parameter group $(U_{h,t})_{t\in\R}$ in $\GL(\cH^\infty)$. 
	Then $(U,\cY)$ is of exponential growth, as noted in the proof of Theorem~\ref{extcl}.  
	Then, for 
	$\phi_n(t) := \gamma_{1/n,0}(t)=\sqrt{\frac{n}{\pi}} e^{-n t^2}$,
	the sequence 
	\[ v_n := U_h(\phi_n) = \sqrt{\frac{n}{\pi}} \int_\R e^{-n t^2} U_{h,t}v \, dt \]
	consists of smooth vectors 
	and $\lim\limits_{n\to\infty}v_n= v$ in~$\cH^\infty$ 
	by Proposition~\ref{agrowth}\eqref{agrowth_item2}. 
	As $\sV'$ is $U_h$-invariant, we also have $v_n \in \sV'$.
	It therefore suffices to show that
	$\eta(v_n) \in \R$ for $n \in \N$.
	
	The advantage of passing to $v_n$ is that its $U_h$-orbit map
	extends to a holomorphic map $U_h^{v_n} \colon  \C \to \cH^\infty$.
	We consider the function
	\[ f \colon \oline{\cS}_{0,\pi} \to \C, \quad
	f(z) := \la U_h^{v_n}(\oline z), U^{-\infty,\eta}_h(z) \ra \]
	which is continuous and holomorphic on $\cS_{0,\pi}$
	by Lemma~\ref{lem:cartes}. 
	As $f\res_\R$ is constant,
	it then follows that $f$ is constant, and thus
	\[ \la v_n, \eta \ra = f(0) = f(\pi \ie)
	= \la J v_n, J \eta \ra = \oline{\la v_n, \eta\ra}.
	\]
	Here we used that the modular operator of the standard subspace $\sV'$ is 
	given by	$\Delta_{\sV'} = \Delta_{\sV}^{-1}$ 
	while $J_{\sV'} = J_{\sV}=J$ 
	(\cite[Lemma~3.7(ii)]{NO17}) 
	to see that, since $v_n\in\sV'$, we have  
	\[ U^{v_n}_h(-\pi \ie)=\Delta_{\sV}^{-1/2}v_n=\Delta_{\sV'}^{1/2}v_n=Jv_n\] 
	(cf.~\cite[Rem.~3.3]{NO17}). 
\end{proof}

\begin{proof}[Proof of Theorem~\ref{thm:vdense}] 
	Clearly, $\sV \subeq \cH^{-\infty}_{\rm KMS}$ is a real subspace.
	To see that it is dense (in the weak-$*$-topology), we have
	to show that, for every $w \in \cH^\infty$ for which
	$f_w := \Re\la w, \cdot \ra$ vanishes on~$\sV$,
	this functional also vanishes on~$\cH^{-\infty}_{\rm KMS}$.
	
	That $f_w(\sV) = \{0\}$ means that $\Re \la w, \sV \ra = \{0\}$,
	so that $\ie w \in \sV'$ and thus Lemma~\ref{lem:2.11} implies that
	$\la \ie w, \cH^{-\infty}_{\rm KMS} \ra \subeq \R$, which in turn shows that
	$f_w$ vanishes on $\cH^{-\infty}_{\rm KMS}$.
\end{proof}

\begin{cor} The annihilator $(\sV' \cap \cH^\infty)^{\bot_\omega}$
	in $\cH^{-\infty}$, with respect to
	$\omega := \Im \la \cdot, \cdot \ra$,
	coincides with~$\cH^{-\infty}_{\rm KMS}$.
\end{cor}

\begin{proof} As $\cH^{-\infty}_{\rm KMS}$ is closed in $\cH^{-\infty}$ and
	$\sV$ is dense in $\cH^{-\infty}_{\rm KMS}$
	by Theorem~\ref{thm:vdense}, we have
	\[ \cH^{-\infty}_{\rm KMS}
	= \{ \eta \in \cH^{-\infty} \colon (\forall \xi \in \sV^{\bot_\R} \cap \cH^\infty)\
	\Re \eta(\xi) = 0\}.\] 
	So it suffices to observe that
	$\sV^{\bot_\R} = (\ie\sV)^{\bot_\omega} = \ie \sV^{\bot_\omega} = \ie \sV'$
	to see that
	\[ \cH^{-\infty}_{\rm KMS}
	= \{ \eta \in \cH^{-\infty} \colon (\forall \xi \in \sV' \cap \cH^\infty)\
	\Im\eta(\xi) = 0\} = (\sV' \cap \cH^\infty)^{\bot_\omega}.\qedhere\] 
\end{proof}

\begin{rem} The preceding corollary shows in particular that
	$\cH^{-\infty}_{\rm KMS}$ is a proper subspace of $\cH^{-\infty}$
	if and only if $\sV' \cap \cH^\infty$ is non-zero.
	Applying $J$, we see that this is equivalent to
	$\sV \cap \cH^\infty \not=\{0\}$.
\end{rem}

We don't have a proof for the following assertion, but
  concrete examples and the case $G = \R$ (\cite[Prop.~2.11]{FNO23})
  suggest that it is true. 

\begin{prob} Show that, for $\eta \in \cH^{-\infty}_{\rm KMS}$,
	the values $U^{-\infty,\eta}_h(z)$, $z \in \cS_\pi$,  are contained in $\cH$.
\end{prob}

\section{Domains of analytic vectors for unitary representations}
\label{Sect7}

In this section we take a closer look at the situations 
arising in the context of the subspace of analytic vectors
$\cH^\omega$ of a unitary representation $(U,\cH)$ of a
Lie group. The space $\cH^\omega$ carries a natural 
topology as a  locally convex direct limit and its antidual
$\cH^{-\omega}$ is a Fr\'echet space
(see \cite[\S 2.1]{FNO23} and \cite{GKS11} for more details on these topologies).
Again, every $h \in \g$ defines one-parameter groups
$(U^{\pm \omega}_{h,t})_{t \in \R}$ on $\cH^{\pm \omega}$, but the discussion
in this section shows that these one-parameter groups
are far from being of exponenial growth.
To verify this claim, we start in Section~\ref{Sect7.1} by
recalling some well-known facts on the Maximum Modulus Theorem
on strips. This is used in Section~\ref{Sect7.2} to show that,
if $G = \Aff(\R)_e$ is the affine group of the real line, then
the orbit maps of the dilation group in $\cH^\omega$ never extend to
strips whose width exceeds $\pi$ (Theorem~\ref{thm:ax+b-gen}).
As non-compact semisimple Lie groups 
contain many copies of the group $\Aff(\R)_e$,
restriction to such subgroups yields with
Theorem~\ref{thm:ax+b-gen} natural ``upper bounds''
on the domains to which orbit maps of analytic vectors
could extend analytically. The corresponding conclusions
are formulated in Theorem~\ref{thm:7.9} which
improves Goodman's results from \cite{Go69}.
That our version of these results is optimal
follows from existence results on
orbit maps by Kr\"otz and Stanton (\cite{KSt04}).
We conclude this paper with a brief discussion of
maximal analytic extensions of orbit maps in $\cH^\omega$
for general Lie groups (Section~\ref{Sect7.4}).

\subsection{Preliminaries on holomorphic functions on strips} 
\label{Sect7.1}

On $\C$ we define for $\eps > 0$ the function
\[    m_\eps(z) := \exp(\eps e^z) \quad \mbox{ with } \quad
|m_\eps(z)| = \exp(\eps e^x \cos(y))
\quad \mbox{ for } \quad z = x + iy.\]
It clearly satisfies
\[ |m_\eps(z)| \leq \exp(\eps e^x) \quad \mbox{ for } \quad z \in \C,\]
but only on the strips $\cS_{\pm r}$ with $r < \pi/2$ we also have
\[   |m_\eps(z)| = \exp(\eps e^x \cos(y))
\geq \exp(\eps e^x \cos(r)) \quad \mbox{ for } \quad |\Im z| \leq r,\]
and since $\cos(r) > 0$, we obtain for $\eps' := \eps \cos(r)$ that
\[ \exp(\eps' e^x) \leq |m_\eps(z)| \quad \mbox{ for } \quad
|z| \leq r.\]
We also note that, on the strip $\cS_{\pm \pi/2}$, we have
\begin{equation}
	\label{eq:mepsinv}
	|m_\eps(z)^{-1}| = \exp(-\eps e^x \cos(y)) \leq 1,
\end{equation}
so that $m_\eps^{-1} \in H^\infty(\cS_{\pm\pi/2})$ (the space of bounded
holomorphic functions), but this function
is unbounded on any larger strip. 

Note that the condition in the following proposition is stronger than the
boundedness of the product $m_\eps f$ on $\cS_{\pm \rho}$. 

\begin{prop} \label{prop:a.1}
	Let $\rho > \pi/2$ and $f \in \cO_\partial(\cS_{\pm \rho})$
	be such that, for some $\eps > 0$,
	\[ M := \sup_{z = x + iy\in \cS_{\pm \rho}} |\exp(\eps \ee^x) f(z)| < \infty. \]
	Then $f = 0$.
\end{prop}

\begin{proof}
	For $\delta > 0$ and $\gamma := \frac{\pi}{2\rho} < 1$, we consider
	the $\cO_\partial(\cS_{\pm \rho})$-function defined by
	\[ g(z) := \exp(\delta \ee^{\gamma z}) f(z).\]
	For $z = x \pm i \rho \in \partial \cS_{\pm \rho}$, we then have
	$\ee^{\gamma z} = \ee^{\gamma x} \ee^{\pm \pi \ie/2} = \pm \ie \ee^{\gamma x} \in \ie \R,$ 
	so that
	\[ |g(z)| = |f(z)| \leq \exp(-\eps \ee^x) M \leq M.\]
	We claim that $g$ is bounded on $\cS_{\pm \rho}$.
	In fact, for $z = x + i y
	\in \cS_{\pm \rho}$, we have
	\begin{align*} |g(z)|
	& = \exp(\Re\delta \ee^{\gamma z}) |f(z)| \\
	& \leq \exp(\delta \ee^{\gamma x} \cos(\gamma y)) \exp(-\eps \ee^x) M \\
	&\leq  \exp(\delta \ee^{\gamma x} - \eps \ee^x) M.
\end{align*}
	Therefore it suffices to observe that the function 
	$x \mapsto \delta \ee^{\gamma x} - \eps \ee^x$ 
	is bounded from  above for every $\delta > 0$. For sufficiently large
	$x$ it is negative, and for $x \to -\infty$ it tends to ~$0$.
	
	Now the Phragm\'en--Lindel\"of Theorem for strips 
	(\cite[Thm.~5.1.9]{Si15}) applies (with
	the constant $\alpha =0$) and shows that 
	\[ \sup_{z \in \cS_{\pm \rho}} |g(z)|
	=  \sup_{z \in \partial\cS_{\pm \rho}} |g(z)| \leq M.\]
	We conclude that, for $z = x \in \R$, 
	\[ |f(x)| \leq \ee^{-\delta \ee^{\gamma x}} M \quad \mbox{ for }  \quad \delta > 0.\]
	For $\delta \to \infty$, we thus obtain $f\res_\R = 0$ and hence
	$f = 0$.  
\end{proof}

\begin{cor} \label{cor:a.1}
	Let $r > \pi/2$ and $f \in H^2(\cS_{\pm r})$.
	If, for some $\rho  \in (\pi/2,r)$, some $\eps > 0$,
        and some $M_\rho>0$, 
	\begin{equation*}
		\sup_{ z = x + iy\in \oline{\cS}_{\pm \rho}}
		\ |\exp(\eps \ee^x) f(z)| \leq M_\rho, 
	\end{equation*}
	then $f = 0$.
\end{cor}

\begin{proof} For any 
	$\rho \in (\pi/2, r)$, the restriction of 
	$f \in H^2(\cS_{\pm r})$ to $\oline{\cS}_{\pm \rho}$ is a continuous
	function, so that Proposition~\ref{prop:a.1} applies.
\end{proof}

\begin{rem} In \cite[Thm.~5.1.9]{Si15}
	the   Phragm\'en--Lindel\"of Theorem for strips
	is formulated for the strip $\cS_1 = \{ z \in \C \colon 0 < \Im z < 1\}$.
	It  asserts that any
	$\cO_\partial$-function on this strip for which there exists
	an $A \in \R$ and $\alpha < \pi$, with
	\begin{equation}
		\label{eq:gcond1}
		|f (z)| \leq  \exp(A \exp(\alpha |z|)) \quad \mbox{ for } \quad
		\Im z \in (0,1), 
	\end{equation}
	is bounded with
	\begin{equation*}
		\sup_{z \in \cS_1} |f (z)| =\sup_{z \in \partial \cS_1} |f (z)|.
	\end{equation*}
	
	For the strip $\cS_\beta$ the critical exponent $\alpha$ must be
	$< \frac{\pi}{\beta}$, as one argues by applying the above result 
	to the function $F(z) := f(\beta z)$ on $\cS_1$. Then
	$F$ satisfies \eqref{eq:gcond1} with $\alpha < \pi$ if and only if
	\begin{equation*}
		|f (z)| \leq  \exp(A \exp(\alpha |z|/\beta)) \quad \mbox{ for } \quad
		\Im z \in (0,\beta), 
	\end{equation*}
	and we have $\alpha/\beta < \pi/\beta$. As
	$|x| \leq |z| \leq |x| + \beta$, one can replace
	$|z|$ by  $x= \Re z$ in these estimates. This shows in particular
	that, for general strips $\cS_{a,b} = \{  z \in\C \colon a < \Im z < b\}$,
	the critical exponent is $\frac{\pi}{b-a}$. 
\end{rem}

\subsection{Unitary representations of $\Aff(\R)$} 
\label{Sect7.2}

Let $G = \Aff(\R)_e$ be the connected affine group  of the real line.
It Lie algebra has a basis $h,y$ satisfying $[h,y] = y$,
and, identifying $G$ with $\R \rtimes \R_+$, we have
$(s,\ee^t) = \exp(sy) \exp(th)$.
We consider the unitary representation $(U,\cH)$ of
$G$ on $\cH = L^2(\R)$ by 
\begin{equation}
	\label{eq:affonl2}
	(U(s,\ee^t)f)(x) = \exp(\ie s \ee^x) f(x+ t) \quad \mbox{ for } \quad x,s,t \in \R.
\end{equation}
Then, for $t > 0$, and $U_h(t) = U(\exp th)$, 
\begin{equation}
	\label{eq:pw}
	\cH^\omega_t(U_h) := \bigcap_{|s| < t} \cD(\ee^{\ie s \partial U(h)})
	= \bigcap_{0 < s < t} H^2(\cS_{\pm s})
	\subeq \cO(\cS_{\pm t})
\end{equation}
(\cite[Prop~5.1]{Go69}).

The operator $\ie\partial U(y)$ is given by multiplication with the
function $- \ee^x$, so that, for $s \in \R$, 
\begin{equation}
	\label{eq:cDy}
	\cD(\ee^{s\ie \partial  U(y)})
	=\Big\{ f \in L^2(\R) \colon  \int_\R \ee^{-2s \ee^x} |f(x)|^2\, dx < \infty\Big\}.
\end{equation}
More precisely, $\cD(\ee^{s\ie \partial  U(y)}) = L^2(\R)$ for $s \geq 0$, but,  
for $s < 0$, this is a proper subspace of $L^2(\R)$.

\begin{rem} From 
	$ U_h(t) \cD(\ee^{s\ie \partial  U(y)})  = \cD(\ee^{s\ee^t \ie \partial  U(y)})$ 
        we get
	\[ U_h(t) \cH^\omega_s(U_y) = \cH^\omega_{\ee^ts}(U_y)\]
	so that $U(\exp \R h)$ preserves the subspace
	\[  \cH^\omega(U_y) = \bigcup_{s > 0} \cD(\ee^{-s  \ie \partial  U(y)}).\]
\end{rem}

\begin{theorem} \label{thm:ax+b-ex}
	Let $f \in \cH^\omega(U)$ be an analytic vector for $U$. 
	Let $r  > \frac{\pi}{2}$ be such that 
	$f \in \cD(\ee^{\pm \ie r \partial U(h)})$  
	and 
	\[ U^f_h(\ie t) = \ee^{\ie t \partial U(h)}f \in \cH^\omega(U) \quad \mbox{ for } \quad
	t < r.\]
	Then $f = 0$. 
\end{theorem}

\begin{proof} For $|t| \leq r$, we put $f_t := U^f_h(\ie t)$ and
pick some 	$\rho \in (\pi/2, r)$. By assumption, 
	$f_t \in \cH^\omega(U)$, so that there exists a $\delta \in (0,\pi/2)$ such that 
	the orbit map $U^{f_t} \colon G = \Aff(\R)_e \to \cH$ extends to a 
	holomorphic map on the open subset
	\[ \exp(B_{3\delta}(0)h) \exp(B_{3\delta}(0)y) \subeq G_\C
          \cong \C \rtimes \C,\]
	where $B_r(0)$ denotes the open disc of radius $r$ in $\C$, and 
	\[ U^{f_t}(\exp(\ie s h)\exp(-\ie \delta y)) 
	= \ee^{\ie s \partial U(h)} \ee^{-\ie \delta \partial U(y)} f_{t}
	\quad \mbox{ for } \quad |s| < 3 \delta. \] 
	Next we observe that
	\[ (\ee^{-\ie \delta \partial U(y)} f_{t})(x)
	= \exp(\delta \ee^x) f_t(x),\]
	so that
	\[ U^{f_t}(\exp(\ie s h)\exp(-\ie \delta y))(x)
	= \exp(\delta \ee^{x+ \ie s}) f_t(x+\ie s).\]
	We conclude with \eqref{eq:pw} that
	\[ \ee^{-\ie \delta \partial U(y)} f_{t}  \in H^2(\cS_{\pm 2\delta}),\]
	so that this function is bounded on the closed strip
	$\oline\cS_{\pm \delta}$ (\cite[Lemma~5.1]{Go69}). 
	Hence there exists an $M_t >0$ such that 
	\begin{align*}
		 M_t 
		 & \geq \sup_{x \in \R, |s| \leq \delta} \exp(\delta \ee^x \cos(s)) |f_t(x+\ie s)| \\
	     & \geq \sup_{x \in \R, |s| \leq \delta} \exp(\delta \ee^x \cos(\delta)) |f(x+\ie(s+t))|.
	\end{align*}
	Thus, for every $t\in(-r,r)$, there exists
	$\delta\in(0,\pi/2)$ (depending on $t$), such that the above estimate holds 
	for $s$ in the interval $[t-\delta,t+\delta]$.
	
	Covering the compact interval
	$[-\rho,\rho]$ with finitely many intervals of the form 
	$[t_j-\delta_j, t_j+\delta_j]$ as above, we obtain for 
	$\delta := \min \{ \delta_j \}$ 
	and $\eps := \delta \cos(\delta)$ an estimate of the form  
	\[ \sup_{z = x + iy\in \cS_{\pm \rho}} \exp(\delta \ee^x \cos(\delta)) |f(z)|
	= \sup_{z = x + iy\in \cS_{\pm \rho}} \exp(\eps  \ee^x ) |f(z)| < \infty.\]
	Now Corollary~\ref{cor:a.1} implies that $f = 0$.
\end{proof}

\begin{ex} \label{ex:7.6} 
	To see that the critical value $r = \pi/2$
	in Theorem~\ref{thm:ax+b-ex} is sharp,
	we have to find some
	$f \in \cH^\omega$ for which the orbit map $U_h^f$ extends to the open strip 
	$\cS_{\pm \pi/2}$ and whose orbit map is contained in $\cH^\omega$.
	
	Let $f = m_\eps^{-1} F$ for some $F \in H^2(\cS_{\pm \pi/2})$
	and recall that $f \in H^2(\cS_{\pm \pi/2})$
	because $m_\eps^{-1}$ is bounded on $\cS_{\pm \pi/2}$ by \eqref{eq:mepsinv}.
	We claim that $f \in \cH^\omega(U)$.
	First, \eqref{eq:pw} implies that $f\in \cD(\ee^{\ie t \partial U(h)})$
	for $|t| < \pi/2$, so that $U^f_h(\ie t)$ is defined for $|t| < \pi/2$.
	Further, $F = m_\eps f \in H^2(\cS_{\pm \pi/2})$. 
	For $r < \pi/2$, this
	implies that $F \in \cO_\partial(\cS_{\pm r})$ is a bounded function
	and that the family
	\[ F_t(x) := F(x + \ie t), \quad |t| \leq r  \]
	is bounded in $L^2(\R)$. 
	As 
	\[ |F_t(x)|^2
	= \exp(2\eps \ee^x \cos(t)) |f(x+\ie t)|^2 
	\geq \exp(2\eps \ee^x \cos(r)) |f(x+\ie t)|^2, \]
	it follows that the function $f_t(x) := f(x + \ie t) = U_h^f(\ie t)(x)$
	satisfies
	\[ \int_\R \exp(2\eps \ee^x \cos(r))|f_t(x)|^2\, \de x < \infty.\] 
	So $f_t \in \cD(\ee^{s \ie \partial U(y)})$ for $|s| \leq \eps \cos(r)$
	by \eqref{eq:cDy}. This implies that $f_t \in \cH^\omega(U)$ for
	$|t| < \pi/2$. 
\end{ex}

\begin{ex}
	With the graph norms 
	\[ \Vert f\Vert_s^2 := \Vert f\Vert^2 + \Vert \ee^{s \ee^t \ie \partial  U(y)} f\Vert^2
	\quad \mbox{ on } \quad
	\cD(\ee^{s \ee^t \ie \partial  U(y)}),\]
	we then have 
	\[ \Vert U_h(t) f\Vert_s=  \Vert f\Vert_{\ee^{-t}s},\]
	showing that the operators $U_h(t)$ on $\cH^\omega(U_y)$ are continuous.
	
	The space $\cH^{-\omega}(U_y)$ of hyperfunction vectors for
	$U_y$ consists of all functions $f \colon  \R \to \C$ for which
	$\ee^{-2s \ee^x} f(x)$ is $L^2$ for all $s > 0$. It is a Fr\'echet
	space with the defining seminorms $(\Vert\cdot\Vert_{-s})_{s > 0}$. 
	The discussion above now shows that
	condition \eqref{agrowth_hyp2} in Definition~\ref{def:expgro}
	is not satisfied for the 
	one-parameter group induced by $U_h$ on $\cH^{-\omega}(U_y)$.
	Here we use that $(U^{-\omega}_{h,t})_{t\in \R}$ is continuous
	on $\cH^{-\omega}$ for the weak-$*$-topology (\cite[Prop.~3.4]{GKS11}).
\end{ex}

\begin{theorem} \label{thm:ax+b-gen}
	Let $(U,\cH)$ be a unitary representation of
	$G = \Aff(\R)_e$ with $\ker(\partial U(y)) = \{0\}$.
	Let $f \in \cH^\omega(U)$ be an analytic vector for $U$. 
	If there exists an $r  > \frac{\pi}{2}$ such that 
	$f \in \cD(\ee^{\pm \ie r \partial U(h)})$  
	and that 
	\[ U^f_h(\ie t) = \ee^{\ie t \partial U(h)}f \in \cH^\omega(U)
          \quad \mbox{ for } \quad
	t < r,\]
	then $f = 0$. 
\end{theorem}

\begin{proof} We consider the two representations
	of $G$ on $L^2(\R)$, given by
	\begin{equation}
		\label{eq:affonl2b}
		(U_\pm(s,\ee^t)f)(x) = \ee^{\pm\ie s \ee^x} f(x+ t),\quad x,s,t \in \R.
	\end{equation}
	Since $\ker(\partial U(y)) = \{0\}$, 
	there are Hilbert spaces $\cK_\pm$ such that the representation
	$(U,\cH)$ is equivalent to the representation 
	\[  (U_+ \otimes \id_{\cK_+}) \oplus (U_- \otimes \id_{\cK_-})  \]
	(\cite{GN47}, \cite[Prop.~2.38]{NO17}). 
	Hence the assertion follows directly from  Theorem~\ref{thm:ax+b-ex}.
\end{proof}

The preceding theorem improves Goodman's \cite[Thm.~7.2]{Go69}
by a factor~$2$. Its assertion corresponds
to Theorem~\ref{thm:ax+b-gen} with $\pi/2$ replaced by~$\pi$.

\subsection{Applications to semisimple Lie groups}
\label{Sect7.3}

We have seen in Theorem~\ref{thm:ax+b-gen} above that,
for any  unitary representation
$(U,\cH)$ of $\Aff(\R)_e$
satisfying the non-degeneracy condition $\ker(\partial U(y)) = \{0\}$,
any analytic vector $v$ for which
$\ee^{\ie t \partial U(h)}v$ is defined and contained
in $\cH^\omega$ for $|t| \leq \pi/2$ is zero.

Now let $G$ be a connected real semisimple Lie group,
let $\g = \fk + \fp$ be a Cartan decomposition 
and $\fa \subeq \fp$ a maximal abelian subspace. 
For a linear functional $0 \not=\alpha \in \fa^*$, the simultaneous eigenspaces 
\[ \g_\alpha :=  \{ y \in \g \colon (\forall x \in \fa) \ [x,y] = \alpha(x)y\} \] 
are called {\it root spaces} and 
\[ \Sigma := \Sigma(\g,\fa) := \{ \alpha \in \fa^* \setminus \{0\}  \colon 
\g_\alpha \not=\{0\}\} \] 
is called the set of {\it restricted roots}.
We then have
\[ \g = \g_0 \oplus  \bigoplus_{\alpha \in \Sigma} \g_\alpha \quad
\mbox{ with } \quad \fa = \g_0 \cap \fp\]
(cf.~\cite[Prop. 6.40]{Kn02}).

\begin{theorem} \label{thm:7.9}
	Let $(U,\cH)$ be a unitary representation of
	the semisimple Lie group $G$ with trivial fixed point space $\cH^G = \{0\}$. 
	If $x \in \g$ is hyperbolic, i.e., $\ad x$ is diagonalizable,
	and $0\not= v \in \cH^\omega$ is such that
	$v \in \cD(\ee^{\ie t \partial U(x)})$ and 
	$\ee^{\ie t \partial U(x)}v \in \cH^\omega$ for $|t| \leq 1$,
	then all eigenvalues of $\ad x$ 
	belong to the interval $[-\pi/2,\pi/2]$. 
\end{theorem}

\begin{proof} By Moore's Theorem \cite[Thm.~1.1]{Mo80},
  $\ker(\partial U(y)) =\{0\}$ for all non-elliptic elements $y \in \g$
  (see also \cite[Lemma 7]{Ma57}).
	Now let $x \in \fa$, $\alpha\in\Sigma$, and $y \in \g_\alpha$ with
	$\alpha(x) \not=0$. 
	If we define $h := \alpha(x)^{-1} x$, 
	then $h\in\fa$ and 
	$[h,y] = y$. 
	For any non-zero analytic vector 
	$v \in \cH^\omega$ as in the statement, we then obtain $v \in \cD(\ee^{\ie \partial U(h)})$ 
	and $U^v_h(t) \in \cH^\omega$ for $|t| \leq \vert\alpha(x)\vert$.  Then Theorem~\ref{thm:ax+b-gen} 
	implies $\vert\alpha(x)\vert\le\pi/2$
	because $\ker(\partial U(y))= \{0\}$. 
	As every hyperbolic element $x \in \g$ is conjugate
	to one in $\fa$ (by \cite[Cor.~II.9]{KN96} or \cite[Cor. 6.19 and Thm. 6.51]{Kn02}), the assertion follows.
\end{proof}

\subsection{Applications to general Lie groups} 
\label{Sect7.4}

In order to deal with general Lie groups~$G$, we start with a
unitary representation $(U,\cH)$ of $G$.
For a non-zero $v \in \cH^\omega$ and $x \in \g$, we consider
the {\it analyticity radius}
\[ r_v(x) := \sup \{ r >  0 \colon  v \in \cD(\ee^{\pm \ie r\partial U(x)}),
(\forall |t| < r) \ \ee^{\pm \ie t\partial U(x)}v \in \cH^\omega\}
\in (0,\infty].\]
To make this independent of $v$, we also consider
\[ r_U(x) := \sup \{ r_v(x) \colon 0\not=v \in\cH^\omega \} \in (0,\infty].\]
We note that if the unitary representation $(U,\cH)$ is norm continuous, 
which is the case for instance if $\dim\cH<\infty$, 
then for every $x\in\g$ the skew-symmetric operator $\partial U(x)\colon\cH\to\cH$ is bounded, 
and this implies
\[ r_v(x)=r_U(x)=\infty \quad \mbox{  for every } \quad
  0\ne v\in\cH^\omega=\cH.\] 

In the following statement we denote by $\rho(T)$ the spectral radius of a
bounded linear operator~$T$. 


\begin{prop} \label{prop:rprops} For a unitary representation $(U,\cH)$
	of the Lie group $G$, the function 
	$r_U  \colon \g \to (0,\infty]$ has the following properties: 
	\begin{enumerate}[\rm(a)]
		\item\label{prop:rprops_item-a}
		$r_U$ is $\Ad(G)$-invariant. 
		\item\label{prop:rprops_item-b}
		$r_U(\lambda x) = |\lambda|^{-1} r_U(x)$ for $\lambda \not=0$.
		\item\label{prop:rprops_item-c}
		For $h, x \in \g$ with $[h,x] = x\not= 0$, 
		we have $r_U(x) = \infty$. 
		\item\label{prop:rprops_item-d}
		Suppose that  $G$ is semisimple and $\cH^G = \{0\}$.
		Then we have: 
		\begin{itemize} 
			\item[\rm(i)] 
			$r_U(x) = \infty$ if $x$ is nilpotent. 
			\item[\rm(ii)] 
			$r_U(h) \leq  \frac{\pi}{2\rho(\ad h)}$
			if $0 \not= h$ is hyperbolic. 
			Equality holds if $G$ is linear and $U$ is irreducible.
			\item[\rm(iii)] 
			If $x$ is elliptic and $U$ is irreducible,
			then $r_U(x) = \infty$.
			\item[\rm(iv)] 
			If $x = x_h + x_e$ is semisimple, where $x_h$ is
                          hyperbolic, $x_e$ elliptic and $[x_h,x_e] = 0$,
                        $G$ is linear, and the representation $U$ is irreducible, 
			then $r_U(x) \geq r_U(x_h)$.
		\end{itemize}
	\end{enumerate}
\end{prop}

\begin{proof} 
	\eqref{prop:rprops_item-a} 
	For $g \in G$ and $v \in \cH^\omega$, the relation
	\[ U(g) \ee^{\ie t \partial U(x)} =  \ee^{\ie t \partial U(\Ad(g)x)} U(g)
	\quad \mbox{  implies  } \quad
	r_{U(g)v}(x)(\Ad(g)^{-1}x) = r_v(x),\]
	which proves 
	\eqref{prop:rprops_item-a}. 
	
	\nin 
	\eqref{prop:rprops_item-b}
	By definition, we have $r_U(-x) = r_U(x)$, so that we may assume that
	$\lambda > 0$. In this case the assertion follows directly from the definition.
	
	\nin 
	\eqref{prop:rprops_item-c}
	The relation $[h,x] = x$ implies $\ee^{t \ad h}x = \ee^t x$, hence
	$r_U(x) = r_U(\ee^tx)$ for $t \in \R$ by (a). As $r_U(x) > 0$, we must have
	$r_U(x) = \infty$.
	
	\nin 
	\eqref{prop:rprops_item-d}(i)
	The Jacobson--Morozov Theorem
	(\cite[Ch.~8, \S 11, no. 2, Prop. 2]{Bo90})  provides an element
	$h \in \g$ with $[h,x] = x$, so that 
	(i) 
	follows from 
	\eqref{prop:rprops_item-c}
	
	\nin (ii) 
	Theorem~\ref{thm:7.9} implies that
	$r_U(h) \leq \frac{\pi}{2\rho(\ad h)}$.
	
	We now assume that $G$ is linear and that $U$ is irreducible.
	We may further assume by 
	\eqref{prop:rprops_item-b} 
	that $\rho(\ad h) = 1$.  
	Let $\g = \fk \oplus \fp$ be a Cartan decomposition and 
	$K = \exp \fk \subeq G$ the corresponding subgroup. 
	Then it follows from  the Kr\"otz--Stanton Extension Theorem 
	(\cite[Thm.~3.1]{KSt04}) that, for every $K$-finite vector
	$v$, the $G$-orbit map extends to an
	analytic domain in $G_\C$ that contains
	$\exp(i\Omega_\fp)$, where
	\[ \Omega_\fp := \{ x \in \fp \colon \rho(\ad x) < \pi/2 \}. \]
	As $th \in \Omega_\fp$ for $|t| < \pi/2$,
	this shows that $r_U(h) \geq \frac{\pi}{2} = \frac{\pi}{2 \rho(\ad h)}$.
	
	\nin (iii) 
	As $x$ is conjugate to an element of $\fk$,
	this follows from Harish--Chandra's Theorem on the
	existence of nonzero $K$-finite vectors (\cite[Thm. 6]{HC53}), 
	since for every $K$-finite vector $v$ we have $r_v(x)=\infty$ 
	by the remark on norm-continuous representations that we made just above the statement. Note that this does
	not require $G$ to be linear nor $K$ to be compact.
	
	\nin (iv) 
	For $x= x_h + x_e$, we can use the same argument as in 
	  (ii).
	For any $K$-finite vector, it implies that
	$\ee^{t \ie \partial U(x)}v = \ee^{t \ie \partial U(x_h)}\ee^{t \ie \partial U(x_e)}v$ is
	defined and an analytic vector if $|t| \rho(\ad x_h) < \frac{\pi}{2}$. 
\end{proof}

\begin{ex} We consider $\g = \aff(\R)$ with the basis
	$h,y$ satisfying $[h,y] = y$
	and a unitary representation $(U,\cH)$
	with $\ker(\partial U(y)) = \{0\}$.
	From Proposition~\ref{prop:rprops} we obtain $r_U(y) = \infty$ and
	from  Theorem~\ref{thm:ax+b-gen} that 
	$r_U(h) \leq \frac{\pi}{2}$.
	As in the proof of Theorem~\ref{thm:ax+b-gen}, 
	$U$ is a direct sum of representations equivalent to
	the representations $U_\pm$ on $L^2(\R)$, defined by \eqref{eq:affonl2b}, 
	hence it follows from Example~\ref{ex:7.6} that
	$r_U(h) = \pi/2$. Now $\Ad$-invariance of $r_U$ and
	Proposition~\ref{prop:rprops}\eqref{prop:rprops_item-b} 
	permit us to calculate
	it on all of $\g$, which results in:
	\begin{equation*}
		r_U(t h + s y) =
		\begin{cases}
			\infty & \text{ for } t = 0 \\
			\frac{\pi}{2t} & \text{ for } |t| \not= 0.  
		\end{cases}
	\end{equation*}
\end{ex}

\begin{theorem} \label{thm:7.12}
	Let $(U,\cH)$ be a unitary representation of a
	Lie group $G$, $h \in \g$ a non-central Euler element,
	i.e., $(\ad h)^3 = \ad h$,  and
	\begin{equation}
		\label{eq:reg-con}
		\bigcap \{ \ker(\partial U(x)) \colon x \in\g_1(h) \cup \g_{-1}(h)\}  = \{0\}.
	\end{equation}
	Further, let $J$ be a conjugation on $\cH$ with
	\[ J U(\exp x) J = U(\exp \ee^{\pi \ie \ad h}x) \quad \mbox{ for } \quad x \in \g.\] 
	Then $(\cH^\omega)_{\rm KMS} = \{0\}$
        with respect to the $\R$-action defined by
          $U_h(t) = U^{-\infty}(\exp th)$.   
\end{theorem}

\begin{proof} Let $v \in (\cH^{\omega})_{\rm KMS}$. Then
	the orbit map $U^v_h \colon \R \to \cH^\omega$ extends to a continuous
	map $\oline\cS_{\pi} \to  \cH^\omega$.
	Let
	\[ w := U^v_h(\pi \ie/2)  \]
	and note that the relation
	$Jv = U^v_h(\pi \ie)$ implies that $Jw = w$. Now
	\[ U^w_h(\ie t) = U^v_h((t + \pi/2) \ie) \in \cH^\omega
	\quad \mbox{ for }\quad |t| \leq \pi/2. \]
	As the endpoints $v$ and $Jv$ of this curve are analytic vectors,
	it can be extended to a curve
	\[ U_h^w \colon [-r,r] \ie \to \cH^\omega \]
	for some $r > \pi/2$.
	In view of Theorem~\ref{thm:ax+b-gen}, applied to the
	$2$-dimensional subalgebras generated by some $y \in \g_1(h) \cup \g_{-1}(h)$
	and $h$, this is only possible if
	$v \in \ker(\partial U(y))$ for all these elements.
	Now~\eqref{eq:reg-con} yields~$v = 0$. 
\end{proof}

\begin{prob} In the preceding proof we made essential use of the fact that 
	the vectors   $\ee^{\ie t \partial U(h)}v$ were $G$-analytic for $|t| \leq \pi/2$.
	In view of the identity $\sV = \cH_{\rm KMS}$
        (see \eqref{eq:v-hkms}), it
	would be a stronger result if we could actually show that 
	\[ \sV \cap \cH^\omega = \{0\}, \quad \mbox{ or even that} \quad
	\sV \cap \cH^\omega = (\cH^\omega)_{\rm KMS}. \]
\end{prob}

\appendix

\section{Weak versus strong holomorphy} 
\label{AppA}

This appendix includes a few results that show that the weak holomorphy property involved in the description of the linear subspace $\cD(U_z)\subseteq\cY$ in Definition~\ref{holomext_def} is  equivalent to the stronger holomorphy property 
that requires the existence of the complex derivatives in may cases of interest. 
For instance, in the setting of Section~\ref{Sect5}, 
this is the case for the Fr\'echet space of smooth vectors  $\cY=\cH^\infty$ or for its antidual space of distribution vectors $\cY=\cH^{-\infty}_c$, endowed with the topology of uniform convergence on compact sets.

\begin{rem}  \label{rem:a.1}
  Let
  $\cX$ be a Hausdorff locally convex space over $\C$ with its topological dual space $\cX'$,
	$\Omega\subseteq\C$ an open subset, and
	$f\colon\Omega\to\cX$ a function. 
	If $\cX$ is sequentially complete, then the following conditions are equivalent: 
	\begin{enumerate}[{\rm(a)}]
		\item The function $f$ is \emph{weakly holomorphic}, i.e., 
		for every $\eta\in\cX'$ the scalar
		function $\eta\circ f\colon \Omega\to\C$ is holomorphic. 
		\item The function $f$ is \emph{holomorphic}, i.e., the limit $f'(z_0):=\lim_{z\to z_0}\frac{f(z)-f(z_0)}{z-z_0}$ exists in $\cX$ for every $z_0\in\Omega$. 
	\end{enumerate}
	See \cite[Th. 3.2]{BS71}, \cite[Th. 2.1.3]{He89}, and \cite[Prop. 2.1.6]{GN}.
\end{rem}

\begin{rem} \label{rem:a.2}
	If $\cY$ is either a Fr\'echet space or an LF space, then $\cY^\sharp_c$ is a complete space, in particular sequentially complete,
	see \cite[Th. 32.2 and its Cor. 4]{Tr67}.
\end{rem}

\begin{rem} \label{rem:a.3}
	If $\cY$ is a barreled space, then 
	$\cY^\sharp$ is quasi-complete 
	(by \cite[Th. 34.2 and its Cor. 2]{Tr67}) 
	and in particular sequentially complete, 
	while $\cY^\sharp_c$ is sequentially complete  
	by \cite[Th. 33.1 and its Cor.]{Tr67}. 
\end{rem}

\section{Continuity of  the antidual actions} 
\label{AppB}

In this appendix we record some continuity properties of group representations that apply in particular to the representation $(U^{-\infty},\cH^{-\infty})$ of the Lie group $G$ from Hypothesis~\ref{concr_hyp}. 

Unless otherwise specified, $G$ is a locally compact group, $\cY$ is a Hausdorff locally convex space over $\C$, and $\pi\colon G\to\GL(\cY)$, $g\mapsto \pi(g)$, is a group morphism.

\begin{prop}
	\label{App_cont1}
	The group representation $\pi$ is continuous,
	in the sense that its corresponding action map 
	\[ \pi^\wedge\colon G\times\cY\to\cY, \quad
	\pi^\wedge(g,y):=\pi(g)y,\]
      is continuous, if and only if  the following two conditions are satisfied:
\begin{enumerate}[{\rm(i)}]
\item The representation $\pi$ is \emph{locally equicontinuous}, 
		i.e., for every compact subset $K\subseteq G$, the set of operators $\pi(K)\subseteq\GL(\cY)$ is equicontinuous, 
		or, equivalently, there exists a neighborhood~$W$ of $\1\in G$ for which the set of operators $\pi(W)\subseteq\GL(\cY)$ is equicontinuous, 
\item The representation $\pi$ is \emph{orbit continuous}, i.e., for every $y\in\cY$,  its corresponding orbit mapping $\pi^y\colon G\to\cY$, $ g\mapsto\pi(g)y$, is continuous at $\1\in G$, or, equivalently, $\pi^y\in C(G,\cY)$. 
\end{enumerate}
\end{prop}

\begin{proof} 		See  \cite[Ch. 8, \S 4.1]{Br68}, \cite[Ch.~3]{Moo68},
or \cite[Lemma 1]{Aa70}.
\end{proof}

\begin{ex}
	\label{App_cont2}
	If $(U,\cY)$ is a one-parameter group with exponential growth, then 
	the corresponding group representation $U\colon \R\to\GL(\cY)$ is orbit continuous by 
	the first condition in Definition~\ref{def:expgro} and is locally equicontinuous as a by-product of the proof of Lemma~\ref{lem:comp-cont}. 
	Therefore the group representation $U\colon \R\to\GL(\cY)$ is continuous by Proposition~\ref{App_cont1}. 
\end{ex}

\begin{ex}
	\label{App_cont3}
	If $(U,\cH)$ is a unitary representation of the Lie group $G$ as in  
	Section~\ref{Sect5}, then the representation
	$(U^\infty,\cH^\infty)$ is locally equicontinuous, 
	as follows easily from the estimate \eqref{homog_gen},
	see also \cite[Eq.~(6)]{Ne10}.
	
	On the other hand, we have $U^\infty(\cdot)y\in C^\infty(G,\cH^\infty)$ by \cite[Prop. 1.2]{Po72}, hence the representation $U^\infty$ is  orbit continuous. Proposition~\ref{App_cont1} 
	then shows that the action map
	\[ G\times\cH^\infty\to\cH^\infty, \quad (g,y)\mapsto U^\infty(g)y,\]
	is continuous. This map is actually smooth by \cite[Th.~4.4]{Ne10}.
\end{ex}

For the following result we adapt the method of proof of \cite[Ch. 8, \S 4.3, Prop. 1]{Br68}. 

\begin{prop}
	\label{App_cont4}
	If the representation $\pi\colon G\to\GL(\cY)$
	of the locally compact group $G$ 
	is continuous in the sense that it defines a continuous
	action on $\cY$,  and we define 
	\[\pi^\sharp\colon G\to\GL(\cY^\sharp),\quad g\mapsto\pi(g^{-1})^\sharp, \]
	then the following assertions hold: 
	\begin{enumerate}[{\rm(i)}]
		\item\label{App_cont4_item1} 
		The group representation $(\pi^\sharp,\cY^\sharp_c)$ is continuous. 
		\item\label{App_cont4_item2} 
		The group representation $(\pi^\sharp,\cY^\sharp_b)$ is locally equicontinuous.
	\end{enumerate}
\end{prop}

\begin{proof}
	We first prove that both group representations $(\pi^\sharp,\cY^\sharp_c)$ 
	and  $(\pi^\sharp,\cY^\sharp_b)$ are locally equicontinuous. 
	To this end, we fix an arbitrary compact subset $K\subseteq G$ 
	and we use the fact that the topology of $\cY^\sharp_c$ (respectively $\cY^\sharp_b$) 
	has a base consisting of polars of compact (respectively, bounded) sets, 
	that is, sets of the form 
	\[B^\circ:=\{\xi\in\cY^\sharp :\sup\vert \langle B,\xi\rangle\vert\le 1\}\]
	for a compact (respectively bounded) subset $B\subseteq\cY$. 
	Therefore, in order to show that the set of operators $\pi^\sharp(K)\subseteq\GL(\cY^\sharp)$ is equicontinuous on $\cY^\sharp_c$ (respectively, $\cY^\sharp_b$) we must prove that, for an arbitrary compact
	(respectively bounded) subset $B\subseteq\cY$, there exists another compact (respectively bounded) subset $B_1\subseteq\cY$ with $\pi^\sharp(K)B_1^\circ\subseteq B^\circ$. 
	In fact, for 
	\[B_1:=\pi^\wedge(K^{-1}\times B)=\pi(K^{-1})B\subseteq\cY\]
	the hypothesis that $\pi^\wedge$ is continuous ensures that $B_1$ is compact (respectively bounded). 
	Moreover, for arbitrary $\xi\in\cY^\sharp$, we have 
	\begin{align*}
		\xi\in B_1^\circ
		& \iff \sup \vert \langle \pi(K^{-1})B,\xi\rangle\vert\le 1 \\
	& \iff (\forall k\in K)\ \sup \vert \langle B,\pi(k^{-1})^\sharp\xi\rangle\vert\le 1 \\
	& \iff \pi^\sharp(K)\xi\subseteq B^\circ.
\end{align*}
	In particular, $\pi^\sharp(K)B_1^\circ\subseteq B^\circ$, as needed. 
		Now \eqref{App_cont4_item2} is proved. 

                To complete the proof of \eqref{App_cont4_item1}, it remains 
	(by Proposition~\ref{App_cont1})
	to show that the  group representation $(\pi^\sharp,\cY^\sharp_c)$ is orbit continuous.  
	To this end, we consider an arbitrary compact subset $H\subseteq\cY$ and fix an arbitrary compact neighborhood $W$ of $\1\in G$. 
	Since the set of operators $\pi^\sharp(W)\subseteq\GL(\cY^\sharp_c)$ is equicontinuous 
	and $\lim_{g\to\1}\pi^\sharp(g)=\id_{\cY^\sharp}$ pointwise, 
	it follows by Remark~\ref{lintop}\eqref{lintop_item2} that 
	$\lim_{W\ni g\to\1}\pi^\sharp(g)=\id_{\cY^\sharp}$ uniformly on the compact subset $H\subseteq\cY$. 
	This implies $\lim_{W\ni g\to\1}\pi^\sharp(g)\xi=\xi$ uniformly on $H$ for arbitrary $\xi\in\cY^\sharp$. 
	Thus $\lim_{g\to\1}\pi^\sharp(g)\xi=\xi$ in $\cY^\sharp_c$. 
\end{proof}

\begin{ex}
	\label{App_cont5}
	If $G=(\R,+)$, $\cY=L^1(\R)$ with its usual Banach space topology, and $\pi\colon G\to\GL(\cY)$ is the regular representation, then for every $\xi\in L^\infty(\R)\setminus C(\R,\C)$ its orbit mapping $\pi^\sharp(\cdot)\xi$ fails to be continuous. 
	This shows that the  group representation $(\pi^\sharp,\cY^\sharp_b)$ in  Proposition~\ref{App_cont4}\eqref{App_cont4_item2} may not be continuous. 
\end{ex}

\begin{ex}
	\label{App_cont6}
	If $(U,\cY)$ is a one-parameter group with exponential growth, then 
	it follows by Example~\ref{App_cont2} along with Proposition~\ref{App_cont4}\eqref{App_cont4_item1} that 
	the corresponding antidual group representation $U^\sharp\colon \R\to\GL(\cY^\sharp_c)$ is continuous. 
\end{ex}

\begin{ex}
	\label{App_cont7}
	In  Example~\ref{App_cont3},  
	for every $g\in G$, the linear operator $U^{-\infty}(g)\colon \cH^{-\infty}\to\cH^{-\infty}$ 
	is continuous if $\cH^{-\infty}$ is endowed with the topology $\sigma/c/b$ of uniform convergence on the finite/compact/bounded subsets of~$\cH^\infty$. 
	It then follows by Proposition~\ref{App_cont4} that the representation 
\[ U^{-\infty}\colon G\to\GL(\cH^{-\infty}_c) \]  is continuous.
      \end{ex}

      That this observation extends to the $G$ action
        on the strong dual is more surprising.
      
      \begin{prop} 
The representation on the strong dual 
	\[ U^{-\infty}\colon G\to\GL(\cH^{-\infty}_b) \] 
	defines a continuous action. 
      \end{prop}

      \begin{proof}
	Arguing as
	in the last paragraph of the proof of Proposition~\ref{App_cont4},
	this follows if $\lim_{g\to\1}U^\infty(g)=\id$ 
	holds uniformly on every bounded subset 
	$B \subeq \cH^\infty$. We now show that this is the case.
	Using the notation from Section~\ref{Sect5}, we
	have to show that, for every any $j_1, \ldots, j_k \in \N_0$
	and
        \[ A := \dd U(x_{i_1}) \cdots \dd U(x_{i_k}),\]  we have
	\[  \lim_{y\to 0} \sup_{\xi \in B} \Vert A (U(\exp y) \xi - \xi)\Vert = 0.\]
	With the calculation \eqref{equal_gen} in Section~\ref{Sect5}, we obtain 
	\allowdisplaybreaks
	\begin{align*}
		   \Vert A (U(\exp y) & \xi - \xi)\Vert \\
		=& \Vert U(\exp(-y)) A U(\exp y) \xi - U(\exp(-y)) A\xi \Vert \\
		\leq & \Vert U(\exp(-y)) A U(\exp y) \xi - A\xi\Vert + \Vert A\xi - U(\exp(-y)) A\xi \Vert \\
		= & \Big\Vert A\xi - \sum_{j_1,\dots,j_k=1}^m 
		a_{i_1 j_1}(y)\cdots a_{i_k j_k}(y)
		\de U(x_{j_1})\cdots\de U(x_{j_k})\xi \Big\Vert \\
		&+ \Vert A\xi - U(\exp -y) A\xi \Vert.  
	\end{align*}
	Since the functions $a_{ij}$ are continuous with $a_{ij}(0) = \delta_{ij}$
	and $A B$ is a bounded subset of $\cH^\infty$, the first summand
	converges uniformly to $0$ on $B$ for $y\to 0$. 
	It thus remains
	to show that 
	\begin{equation}
		\label{eq:ts}
		\lim_{y\to 0} \sup_{\xi \in B} \Vert U(\exp y) \xi - \xi\Vert = 0.
	\end{equation}
	From 
	\[  U(\exp y) \xi - \xi = \int_0^1 U(\exp ty) \dd U(y) \xi\, \de t \]
	we derive that
	\[  \Vert U(\exp y) \xi - \xi\Vert  \leq \Vert\dd U(y) \xi \Vert.\]
	Writing $y = \sum_i y_i x_i$ for the expansion of $y$ in the basis
	$x_1,\ldots, x_m$, this leads to 
	\[  \Vert U(\exp y) \xi - \xi\Vert  \leq \sum_{i = 1}^m \vert y_i\vert \cdot
	\Vert\dd U(x_i) \xi \Vert,\]
	from which \eqref{eq:ts} follows.     
        
      \end{proof}

\end{document}